\documentclass[leqno, 12pt, a4paper]{amsart}
\usepackage{amsmath}
\usepackage{amssymb, latexsym, mathrsfs,  mathabx, dsfont, a4wide}
\usepackage{color}

\makeatletter
\@namedef{subjclassname@2020}{\textup{2020} Mathematics Subject Classification}
\makeatother

 \newtheorem{theorem}{Theorem}[section]
 \newtheorem{lemma}[theorem]{Lemma}
 \newtheorem{proposition}[theorem]{Proposition}

 \newtheorem*{theorem*}{Theorem}
\newtheorem*{proposition*}{Proposition}
\newtheorem*{lemma*}{Lemma}

\theoremstyle{definition}
 \newtheorem{definition}[theorem]{Definition}

 \theoremstyle{remark}
 
 \newtheorem{remark}[theorem]{Remark}

   \newtheorem*{claim*}{Claim}

\newcommand{\op}[1]{\operatorname{#1}}

\newcommand{\cf}{\emph{cf}.}

\newcommand{\acou}[2]{\ensuremath{\left\langle #1 , #2 \right\rangle}}

\newcommand{\brak}[1]{\ensuremath{\langle #1\rangle}}

\newcommand{\Tr}{\ensuremath{\op{Tr}}}

\def\XXint#1#2#3{{\setbox0=\hbox{$#1{#2#3}{\int}$}
\vcenter{\hbox{$#2#3$}}\kern-.5\wd0}}

\newcommand{\C}{\ensuremath{\mathbb{C}}} 
 
\newcommand{\N}{\ensuremath{\mathbb{N}}} 
 
\newcommand{\R}{\ensuremath{\mathbb{R}}} 
 
\newcommand{\Z}{\ensuremath{\mathbb{Z}}}

\newcommand{\Rn}{\ensuremath{\R^{n}}}

\newcommand{\cA}{\mathscr{A}}

\newcommand{\cH}{\ensuremath{\mathscr{H}}}

\newcommand{\cL}{\ensuremath{\mathscr{L}}}

\newcommand{\cS}{\ensuremath{\mathscr{S}}}

\newcommand{\tS}{\textup{S}}
\newcommand{\tCS}{\textup{CS}}
\newcommand{\tL}{\textup{L}}
\newcommand{\tCL}{\textup{CL}}

\def\dba{{\mathchar'26\mkern-12mu d}}
\newcommand{\dbar}{{\, \dba}}

\newcommand{\Sp}{\op{Sp}}

\newcommand{\bU}{\pmb{U}}

\numberwithin{equation}{section}

\begin{document}

\title{Weakly Parametric Pseudodifferential Calculus for Twisted $C^*$-dynamical Systems}

 \author{Gihyun Lee}
 \address{Department of Mathematics: Analysis, Logic and Discrete Mathematics, Ghent University, Krijgslaan 281, Building S8, B 9000 Ghent, Belgium}
 \email{Gihyun.Lee@UGent.be}
 
 \author{Matthias Lesch}
 \address{Mathematisches Institut, Universit\"at Bonn, Endenicher Allee 60, 53115 Bonn, Germany}
 \email{lesch@math.uni-bonn.de}

 \thanks{GL was supported by the Max Planck Institute for Mathematics, Bonn, the FWO Odysseus 1 grant G.0H94.18N: Analysis and Partial Differential Equations, and the Methusalem programme of the Ghent University Special Research Fund (BOF) (Grant number 01M01021).
ML was supported by the Hausdorff Center for Mathematics, Bonn}
 
 
 
\begin{abstract}
For a twisted $C^*$-dynamical system $(\cA,\Rn,\alpha,e)$ over a unital $C^*$-algebra we establish a weakly parametric pseudodifferential calculus analogously to the celebrated weakly parametric
calculus due to Grubb and Seeley~\cite{GS:IM95}. If the $C^*$-algebra $\cA$ has an $\alpha$-invariant trace then we prove an expansion of the resolvent trace (with respect to the dual trace on multipliers) for suitable pseudodifferential multipliers. The question whether the expansion holds true as a Hilbert space trace expansion in concrete GNS spaces for $\cA$ will be addressed in a future publication.
\end{abstract}

\maketitle 

\section{Introduction} \label{sec:intro}

The purpose of this paper is to establish a weakly parametric pseudodifferential calculus for twisted $C^*$-dynamical systems. Let us first put this into perspective. Pseudodifferential operators were developed to be able to treat the resolvent of a differential operator as a ``virtual differential operator'' of negative order and hence to a large extent on an equal footing as differential operators. Soon it became clear from Seeley's famous paper on complex powers~\cite{Se:PSPM67} that in order to obtain the full strength of the results on the trace of the $\zeta$-function (or more or less equivalently the short time asymptotic expansion of the heat trace) it was necessary to extend the pseudodifferential calculus to a \emph{calculus with parameters}, \cf\ e.g.,
\cite{Sh:Springer01}. It is important to understand that this means that the resolvent parameter is essentially treated like a covariable and it should not be confused with the naive perception of being just an auxiliary parameter in the theory. However, as it turned out the parametric calculus works well only for resolvents of \emph{differential operators}. This
subtlety was even overlooked in the early days and observed only much later in the paper by Duistermaat and Guillemin~\cite{DG:IM75}. In order to 
incorporate the resolvent of a pseudodifferential operator into the theory it was therefore necessary to establish a \emph{weakly parametric calculus} which was fully worked out much later by Grubb and Seeley~\cite{GS:IM95}. The weakly parametric calculus is indispensable for establishing the more subtle invariants of an elliptic operator as e.g., the noncommutative residue~\cite{Gu:AM85, Wo:PhDThesis, Wo:Springer87}. The latter is intimately related to the $\log t$ term in the asymptotic expansion of $\Tr(A e^{-t\Delta})$, say, as $t\to 0$, it vanishes for a differential operator $A$.

Pseudodifferential calculi are also of significance in the noncommutative setting~\cite{Ba:CRAS88,Co:CRAS80}. On noncommutative tori, which form an interesting model case for a noncommutative space, meanwhile a rich spectral geometry has been established~\cite{CF:Muenster19, CM:JAMS14, CT:Baltimore11, DL:JMP13, FG:SIGMA16, KM:JGP14, LM:GAFA16}. In \cite{LM:GAFA16} a full heat trace asymptotics for Laplace type operators on Heisenberg modules, a class of natural projective modules over noncommutative tori, was established and applied to obtain many of the spectral geometry results known from the conformal geometry of compact oriented surfaces. The main technical tool there was a parametric pseudodifferential calculus (\`a la Shubin~\cite{Sh:Springer01}) for twisted $C^*$-dynamical systems. In a second step this was then used to obtain a heat trace asymptotic for the effective action on Heisenberg modules. The second step is nontrivial as the natural traces on the $C^*$-dynamical system and the Hilbert space trace on the Heisenberg module are only ``asymptotically equal'' (see~\cite[Thm.\ 6.2]{LM:GAFA16}). It should be noted that the Heisenberg modules are the natural noncommutative torus analogues of the classification of holomorphic vector bundles over elliptic curves (see, e.g., \cite{Po:DocMath04}).

If one wants to push the analogy between the spectral geometry of noncommutative tori and compact surfaces further it is therefore natural to extend Grubb and Seeley's weakly parametric calculus to this situation. The current paper is an important step in this direction. Namely, we work out the weakly parametric calculus for abstract twisted $C^*$-dynamical systems. The calculus here is reminiscent of the calculus on $\Rn$ (\cf\ \cite[Chap.\ 4]{Sh:Springer01}). What is missing, however, and which needs to be addressed in a future publication is the above mentioned transfer to the effective action on Heisenberg modules. The obstacle is that so far we are unable to prove the analog of  \cite[Thm.\ 6.2]{LM:GAFA16} about the asymptotic comparison between the abstract trace and the concrete Hilbert space trace for a weakly parametric operator family.

Let us now in some more detail sketch the set-up and main results of the paper. Let $(\cA,\Rn,\alpha,e)$ be a twisted $C^*$--dynamical system, i.e., $\cA$ is a unital $C^*$-algebra and $\alpha$ is a continuous $\Rn$ action by automorphisms. The twisting $e(x,y)=e^{i\acou {Bx} y }$ is given by a skew-symmetric real $n\times n$ matrix. $\cA^\infty$ denotes, as usual, the subalgebra of smooth elements with respect to the action. We now
study pseudodifferential multipliers on the Hilbert $\cA$-module $\cH(\Rn,\cA) := L^2(\Rn)\otimes \cA$. The latter is the exterior tensor product (\cf\ \cite[Chap.\ 4]{La:CUP95}) of the Hilbert space $L^2(\Rn)$ with $\cA$, and here $\cA$ is regarded as a Hilbert module over itself. More concretely $\cH(\Rn,\cA)$ is the completion of the Schwartz space $\cS(\Rn,\cA^\infty)$ with respect to the inner product,
\begin{equation*}
\acou f g = \int_{\Rn} f(x)^*g(x) dx , \qquad f,g\in\cS(\Rn,\cA^\infty) .
\end{equation*}
The twisted dynamical system induces on $\cS(\Rn,\cA^\infty)$ an adjoint~(\ref{eq:PsiDOs.Hilbert-premodule-adjoint}) and a convolution product~(\ref{eq:PsiDOs.Hilbert-premodule-convolution}) which turn it into a $*$-algebra. Furthermore, the convolution product gives rise to a left regular representation of $\cS(\Rn,\cA^\infty)$ by bounded adjointable multipliers on the Hilbert $\cA$-module $\cH(\Rn,\cA)=L^2(\R^n)\otimes\cA$. Extending this left regular representation to more general function classes now leads naturally to ``pseudodifferential multipliers''. For the usual H\"ormander symbol class and $\cA=\C$ this specializes to a well-known pseudodifferential calculus on $\Rn$ (\cf\ \cite[Chap.\
4]{Sh:Springer01}).

In Section~\ref{sec:PsiDOs} we extend the H\"ormander symbol classes to our situation. This is a straightforward generalization of the Baaj-Julg symbol spaces~\cite{Ba:CRAS88} from the untwisted to the twisted case. Furthermore, we work out in some more detail the pseudodifferential calculus for a twisted
$C^*$-dynamical system as indicated in~\cite[\S{3.2}]{LM:GAFA16}. We follow also~\cite{HLP:Part1, HLP:Part2}. For a symbol $f\in\tS^m(\Rn,\cA^\infty)$, $m\in\R$, the pseudodifferential multiplier $P_f$ associated to $f$ is defined by
\begin{equation} \label{eq:PsiDOs.PsiDOs.definition.intro}
    (P_f u)(x) := (M_{f^\vee} u)(x)
      = \frac{1}{(2\pi)^n}\int_{\Rn} e^{i\acou x \xi} \alpha_{-x}(f(\xi+Bx))\hat{u}(\xi)
      d\xi,
     \quad u\in\cS(\Rn,\cA^\infty) .
\end{equation}
The rules of calculus are worked out in Section~\ref{sec:PsiDOs} below. We also put particular emphasis on the asymptotic symbol formulas for \emph{classical} operators which are a bit more involved here due to the twisting (Theorem~\ref{thm:PsiDOs.composition-formula-classical-symbols}).

In Section~\ref{sec:boundedness} we prove
a basic boundedness result for pseudodifferential multipliers on $\cH(\Rn,\cA)$ (Proposition~\ref{prop:boundedness.PsiDOs-boundedness}). As a tool we prove a Hilbert $C^*$-module version of the classical Schur's test on boundedness for integral operators (Lemma~\ref{lem:boundedness.Schur-test}).

In Section~\ref{sec:weakly} we develop the weakly parametric pseudodifferential calculus for pseudodifferential multipliers acting on the Hilbert module $\cH(\Rn,\cA) = L^2(\Rn)\otimes\cA$. This is an adaption of Grubb and Seeley's weakly parametric calculus for ordinary pseudodifferential operators~\cite{GS:IM95}.

The final Section~\ref{sec:Resolvents} contains our main results on the trace expansions, i.e., we have
\begin{theorem}[Theorem~\ref{eq:asymptotic.trace-expansion}]
\label{main.expansion.result.intro}
Let $P$ and $A$ be classical (i.e., $1$-step polyhomogeneous)
pseudodifferential multipliers with respective orders $m\in\N$ and
$\omega\in\R$. Suppose that $P$ is elliptic with parameter
$\mu\in\Gamma$. Then, for $\lambda\in-\Gamma^m$ and $k$ with
$-km+\omega<-n$, we have the asymptotic expansion,
\begin{equation*}
\Tr_\psi\left[ A(P-\lambda)^{-k} \right] \sim \sum_{j=0}^\infty c_j\lambda^{\frac{n+\omega-j}{m}-k} + \sum_{l=0}^\infty \big( c_l'\log{\lambda} + c_l'' \big) \lambda^{-k-l} .
\end{equation*}
Here the coefficients $c_j$, $c_l'$ and $c_l''$ are given by the the integral (over $\Rn$) of the trace $\psi$ of the respective symbols $f(\xi)\sim\sum_{j\geq 0}f_{m-j}(\xi)$ and $a(\xi)\sim\sum_{j\geq 0}a_{\omega-j}(\xi)$ of $P$ and $A$.
\end{theorem}
We refer the reader to Definition~\ref{def:resolvents.elliptic-with-parameter} for the precise meaning of the notion of ellipticity with parameter. Furthermore, $\Tr_\psi$ is the trace on pseudodifferential multipliers induced from an $\alpha$-invariant continuous trace $\psi$ on $\cA$ (\cf\ (\ref{eq:PsiDOs.natural-trace-on-PsiDOs})). If the twisted $C^*$-dynamical system $(\cA,\Rn,\alpha,e)$ is projectively represented on a Hilbert space $\cH$ then one obtains a representation of the pseudodifferential multipliers as operators on $\cH$. In general $\Tr_\psi(P_f)$ does not coincide with the Hilbert space trace of the representation of $P_f$. In~\cite[Thm.\ 6.2]{LM:GAFA16} however, it was shown that for the standard parametric calculus the two traces coincide up to a summand of order $O(\lambda^{-N})$, $N$ arbitrary. This suffices to deduce an asymptotic expansion result analogous to Theorem~\ref{main.expansion.result.intro} also for the Hilbert space realization. For the weakly parametric calculus the method of loc. cit. does not work (see Remark~\ref{rem:asymptotic.two-traces-comparison} for more information on this point) and we have to leave this question for a future publication.

\section{Pseudodifferential Calculus for Twisted $C^*$-dynamical Systems} \label{sec:PsiDOs}
In this section, we recall the main definitions and properties of pseudodifferential multipliers on twisted $C^*$-dynamical systems~\cite{LM:GAFA16, LM:ANCG}. This is a generalization of the pseudodifferential calculus on $C^*$-dynamical systems described in~\cite{Ba:CRAS88, Co:CRAS80}.

\subsection{Twisted $C^*$-dynamical systems} Let $\cA$ be a unital $C^*$-algebra and $\alpha$ be a continuous action of $\Rn$ on $\cA$ by $C^*$-algebra automorphisms, i.e., $\alpha_t\in\textup{Aut}(\cA)$ for all $t\in\Rn$ and the map $t\mapsto\alpha_t(a)$ is norm-continuous for all $a\in\cA$. Furthermore, let
\begin{equation*}
e(x,y) := e^{i\acou {Bx} y}, \qquad x,y\in\Rn ,
\end{equation*}
where $B = (b_{kl})$ is a skew-symmetric real $n\times n$ matrix and $\acou \cdot \cdot$ is the standard inner product on $\Rn$. Then the quadruple $(\cA,\Rn,\alpha,e)$ forms a twisted $C^*$-dynamical system.

Throughout the paper, we let $(\cA,\Rn,\alpha,e)$ be the twisted $C^*$-dynamical system described above. We shall denote by $\cA^\infty$ the smooth subalgebra of $\cA$ induced by the action $\alpha$, i.e., those $a\in\cA$ such that the map $\Rn\ni t\mapsto\alpha_t(a)\in\cA$ is smooth. For $\gamma\in\N_0^n$, we define
\begin{equation*}
\delta^\gamma a := i^{-|\gamma|}\partial_t^\gamma\big|_{t=0} \alpha_t(a) , \qquad a\in\cA^\infty.
\end{equation*}
$\cA^\infty$ is a Fr\'echet space with respect to the locally convex topology generated by the semi-norms $a\mapsto\|\delta^\gamma a\|$, $\gamma\in\N_0^n$.

Let us denote by $\cS(\Rn,\cA^\infty)$ the space of Schwartz class maps with values in $\cA^\infty$. Let $f\in\cS(\Rn,\cA^\infty)$. Its Fourier transform $\hat{f}:\Rn\rightarrow\cA^\infty$ is defined by
\begin{equation*}
\hat{f}(\xi) = \int_{\Rn} e^{-i\acou x \xi} f(x) dx , \qquad \xi\in\Rn .
\end{equation*}
Furthermore, the inverse Fourier transform $f^\vee:\Rn\rightarrow\cA^\infty$ of $f\in\cS(\Rn,\cA^\infty)$ is defined by
\begin{equation*}
f^\vee(x) = \int_{\Rn} e^{i\acou \xi x} f(\xi) \dbar\xi , \qquad x\in\Rn ,
\end{equation*}
where we have set $\dbar\xi := (2\pi)^{-n}d\xi$. The Fourier transform and the inverse Fourier transform induce continuous linear isomorphisms on $\cS(\Rn,\cA^\infty)$ that are inverses of each other. We refer to~\cite[Appendix~B and Appendix~C]{HLP:Part1} for a more detailed account on the integration and the Fourier transform of a map with values in locally convex spaces.

We can endow the Schwartz space $\cS(\Rn,\cA^\infty)$ with the pre-$C^*$-module structure given by the inner product,
\begin{equation} \label{eq:PsiDOs.Cstar-module-inner-product}
\acou f g = \int_{\Rn} f(x)^*g(x) dx , \qquad f,g\in\cS(\Rn,\cA^\infty) .
\end{equation}
For $f\in\cS(\Rn,\cA^\infty)$, set
\begin{gather*}
(af)(x) = \alpha_{-x}(a)f(x) , \qquad a\in\cA^\infty , \\
(\bU_y f)(x) = e(x,-y)f(x-y) .
\end{gather*}
Then $\bU_y$, $y\in\Rn$, is a projective family of unitaries such that, for all $x,y\in\Rn$, we have
\begin{gather*}
\bU_x^* = \bU_{-x} , \\
\bU_x\bU_y = e(x,y)\bU_{x+y} , \\
\bU_x a \bU_{-x} = \alpha_x(a) , \qquad a\in\cA^\infty .
\end{gather*}
Given $f\in\cS(\Rn,\cA^\infty)$, we define the multiplier $M_f$ associated with $f$ by
\begin{equation*}
M_f = \int_{\Rn}f(x)\bU_x dx .
\end{equation*}
Then the space $\cS(\Rn,\cA^\infty)$ becomes a $*$-algebra with respect to the product and the adjoint defined by $M_f\circ M_g = M_{f*g}$ and $M_f^* = M_{f^*}$, where
\begin{gather}
\label{eq:PsiDOs.Hilbert-premodule-adjoint} f^*(x) = \alpha_x\left( f(-x)^* \right) , \\
\label{eq:PsiDOs.Hilbert-premodule-convolution} (f*g)(x) = \int_{\Rn} f(y)\alpha_y \big( g(x-y) \big)e(y,x) dy .
\end{gather}

Suppose that the algebra $\cA$ is equipped with an $\alpha$-invariant continuous trace $\psi$. In this case, $\psi$ induces the \emph{dual trace} on $\cS(\Rn,\cA^\infty)$ which is given by
\begin{equation} \label{eq:PsiDOs.dual-trace-definition}
\widehat{\psi}(f) := \psi\big( f(0) \big) = \int_{\Rn} \psi\big( \hat{f}(\xi) \big) \dbar\xi , \qquad \dbar\xi := (2\pi)^{-n}d\xi .
\end{equation}

\subsection{Symbols and pseudodifferential multipliers} \label{subsec:symbols-and-PsiDOs}
\begin{definition}[\cite{Ba:CRAS88, Co:CRAS80}] \label{def:PsiDOs:standard-symbols}
The symbol space $\tS^m(\Rn,\cA^\infty)$, $m\in\R$, consists of smooth maps $f:\Rn\rightarrow\cA^\infty$ such that, for all $\alpha,\beta\in\N_0^n$, there exists $C_{\alpha\beta}>0$ such that
\begin{equation*}
\big\| \delta^\alpha\partial_\xi^\beta f(\xi) \big\| \leq C_{\alpha\beta} \brak{\xi}^{m-|\beta|} ,
\end{equation*}
for all $\xi\in\Rn$. Here we denote $\brak{\xi} := (1+|\xi|^2)^{\frac{1}{2}}$.
\end{definition}

We endow $\tS^m(\Rn,\cA^\infty)$, $m\in\R$, with the locally convex topology generated by the semi-norms,
\begin{equation} \label{eq:PsiDOs.standard-symbol-semi-norms}
p_N(f) := \sup_{|\alpha|+|\beta|\leq N} \sup_{\xi\in\Rn} \, \brak{\xi}^{-m+|\beta|} \big\| \delta^\alpha\partial_\xi^\beta f(\xi) \big\| , \qquad N\in\N_0 .
\end{equation}
The space $\tS^m(\Rn,\cA^\infty)$ is a Fr\'echet space with respect to these semi-norms (see, e.g., \cite[Prop.\ 3.3]{HLP:Part1} for the proof).

\begin{lemma}[{see~\cite[Lem.\ 3.5]{HLP:Part1}}] \label{lem:PsiDOs.pointwise-product-continuity}
Let $m_1,m_2\in\R$. Then the product of $\cA^\infty$ gives rise to a continuous bilinear map from $\tS^{m_1}(\Rn,\cA^\infty)\times\tS^{m_2}(\Rn,\cA^\infty)$ to $\tS^{m_1+m_2}(\Rn,\cA^\infty)$.
\end{lemma}

\begin{definition}[\cite{Ba:CRAS88}]
Let $f\in\tS^m(\Rn,\cA^\infty)$, $m\in\R$, and let $f_j\in\tS^{m-j}(\Rn,\cA^\infty)$, $j = 0,1,\ldots$. We shall write $f(\xi)\sim\sum_{j\geq 0}f_j(\xi)$ when
\begin{equation} \label{eq:PsiDOs.asymptotic-expansion-definition}
f(\xi) - \sum_{j<N}f_j(\xi)\in\tS^{m-N}(\Rn,\cA^\infty) \qquad \text{for all $N\geq 1$} .
\end{equation}
\end{definition}
There is a version of Borel's lemma for $\cA^\infty$-valued symbols.

\begin{lemma}[{see also~\cite[Prop.\ 3.4]{FW:JPDOA11} and~\cite[Lem.\ 3.10]{HLP:Part1}}] \label{lem:PsiDOs.Borel-lemma}
Let $m\in\R$ and $f_j(\xi)\in\tS^{m-j}(\Rn,\cA^\infty)$, $j\geq 0$. Then there exists $f(\xi)\in\tS^m(\Rn,\cA^\infty)$ such that $f(\xi)\sim\sum_{j\geq 0}f_j(\xi)$ in the sense of~\textup{(\ref{eq:PsiDOs.asymptotic-expansion-definition})}.
\end{lemma}

We say that a map $f:\Rn\rightarrow\cA^\infty$ is \emph{homogeneous} of degree $m\in\R$ if
\begin{equation*}
f(\lambda\xi) = \lambda^mf(\xi) \qquad \text{for all $\lambda\geq 1$ and $\xi\in\Rn\setminus B(0,1)$} .
\end{equation*}
Here $B(0,r)$, $r>0$, denotes the open ball of radius $r$ centered at the origin. Note that if $f:\Rn\rightarrow\cA^\infty$ is homogeneous of degree $m\in\R$, then $f\in\tS^m(\Rn,\cA^\infty)$.

\begin{definition} \label{def.PsiDOs.classical-symbols}
The space of classical ($1$-step polyhomogeneous) symbols, denoted by $\tCS^m(\Rn,\cA^\infty)$, $m\in\R$, consists of maps $f(\xi)\in\tS^m(\Rn,\cA^\infty)$ that admit an asymptotic expansion,
\begin{equation*}
f(\xi)\sim\sum_{j\geq 0}f_{m-j}(\xi) ,
\end{equation*}
where $\sim$ is meant in the sense of~(\ref{eq:PsiDOs.asymptotic-expansion-definition}) and $f_{m-j}:\Rn\rightarrow\cA^\infty$ is a smooth homogeneous map of degree $m-j$ for each $j\geq 0$.
\end{definition}

\begin{remark} \label{rem:PsiDOs.classical-symbol-partial-derivative}
Let $f(\xi)\in\tCS^m(\Rn,\cA^\infty)$, $m\in\R$, be such that $f(\xi)\sim\sum_{j\geq 0}f_{m-j}(\xi)$. Then, for every $j\geq 0$ and $\alpha,\beta\in\N_0^n$, $\delta^\alpha\partial_\xi^\beta f_{m-j}(\xi)$ is homogeneous of degree $m-|\beta|-j$. Furthermore, it follows from the very definition of classical symbols that, for all $\alpha,\beta\in\N_0^n$, we have $\delta^\alpha\partial_\xi^\beta f(\xi)\in\tCS^{m-|\beta|}(\Rn,\cA^\infty)$ and $\delta^\alpha\partial_\xi^\beta f(\xi)\sim\sum_{j\geq 0}\delta^\alpha\partial_\xi^\beta f_{m-j}(\xi)$ in the sense of~(\ref{eq:PsiDOs.asymptotic-expansion-definition}).
\end{remark}

Following~\cite[\S{3}]{LM:GAFA16} it can be shown that, for $f,u\in\cS(\Rn,\cA^\infty)$ we have
\begin{equation} \label{eq:PsiDOs:inverse-Fourier-transform-multiplier-formula}
(M_{f^\vee}u)(x) = \int_{\Rn} e^{i\acou x \xi} \alpha_{-x}(f(\xi+Bx))\hat{u}(\xi) \dbar\xi .
\end{equation}
Note that this integral makes sense even for $f\in\tS^m(\Rn,\cA^\infty)$, $m\in\R$, and $u\in\cS(\Rn,\cA^\infty)$.

\begin{definition}[\cite{LM:GAFA16, LM:ANCG}]
Let $f\in\tS^m(\Rn,\cA^\infty)$, $m\in\R$. The pseudodifferential multiplier $P_f$ associated with $f$ is defined by
\begin{equation} \label{eq:PsiDOs.PsiDOs.definition}
P_f u := M_{f^\vee} u , \qquad u\in\cS(\Rn,\cA^\infty) .
\end{equation}
The space of pseudodifferential multipliers of order $m$ is denoted by $\tL_\sigma^m(\Rn,\cA^\infty)$. Furthermore, if $f\in\tCS^m(\Rn,\cA^\infty)$, we say that $P_f$ is a classical pseudodifferential multiplier, and we shall denote the space of classical pseudodifferential multipliers of order $m$ by $\tCL_\sigma^m(\Rn,\cA^\infty)$.
\end{definition}

As mentioned in~\cite{LM:GAFA16} the space of pseudodifferential multipliers forms a $*$-algebra. More precisely, we have the following results.

\begin{theorem}[{Compare~\cite[Thm.\ 3.2]{LM:GAFA16}}] \label{thm:PsiDOs.composition-formula}
Let $f\in\tS^m(\Rn,\cA^\infty)$ and $g\in\tS^{m'}(\Rn,\cA^\infty)$. We set
\begin{equation} \label{eq:PsiDOs.composition-formula}
f\sharp g(\xi) = \int \bigg( \int e^{-i\acou y \eta} f(\eta+\xi)\alpha_{-y} \big( g(\xi+By) \big) dy \bigg) \dbar\eta .
\end{equation}
Then the following holds.
\begin{enumerate}
\item The map $(f,g)\mapsto f\sharp g$ gives rise to a continuous bilinear map from $\tS^m(\Rn,\cA^\infty)\times\tS^{m'}(\Rn,\cA^\infty)$ to $\tS^{m+m'}(\Rn,\cA^\infty)$.
\item We have $P_fP_g = P_{f\sharp g}$.
\item The symbol $f\sharp g\in\tS^{m+m'}(\Rn,\cA^\infty)$ admits the asymptotic expansion,
\begin{equation} \label{eq:PsiDOs.composition-asymptotic-expansion}
f\sharp g(\xi)\sim\sum_\alpha \frac{(-i)^{|\alpha|}}{\alpha !} (\partial_\xi^\alpha f)(\xi) \partial_x^\alpha|_{x=0} \Big( \alpha_{-x} \big( g(\xi+Bx) \big) \Big) .
\end{equation}
Here $\sim$ is taken in the sense of~\textup{(\ref{eq:PsiDOs.asymptotic-expansion-definition})}.
\end{enumerate}
\end{theorem}

\begin{theorem}[{Compare~\cite[Thm.\ 3.2]{LM:GAFA16}}] \label{thm:PsiDOs.adjoint-formula}
Let $f\in\tS^m(\Rn,\cA^\infty)$. We set
\begin{equation} \label{eq:PsiDOs.adjoint-formula}
f^\star(\xi) = \int \bigg( \int e^{i\acou y \eta} \alpha_y \big( f(\eta+\xi)^* \big) dy \bigg) \dbar\eta .
\end{equation}
Then the following holds.
\begin{enumerate}
\item The map $f\mapsto f^\star$ gives rise to a continuous anti-linear map from $\tS^m(\Rn,\cA^\infty)$ to itself.
\item We have $P_f^* = P_{f^\star}$. Here $P_f^*$ is the formal adjoint of $P_f$ with respect to the inner product~\textup{(\ref{eq:PsiDOs.Cstar-module-inner-product})}.
\item The symbol $f^\star\in\tS^m(\Rn,\cA^\infty)$ admits the asymptotic expansion,
\begin{equation} \label{eq:PsiDOs.adjoint-asymptotic-expansion}
f^\star(\xi)\sim\sum_\alpha \frac{1}{\alpha !} \delta^\alpha\partial_\xi^\alpha \left[ f(\xi)^* \right] .
\end{equation}
Here $\sim$ is taken in the sense of~\textup{(\ref{eq:PsiDOs.asymptotic-expansion-definition})}.
\end{enumerate}
\end{theorem}

\begin{remark}
The integrals in~(\ref{eq:PsiDOs.composition-formula}) and~(\ref{eq:PsiDOs.adjoint-formula}) make sense as an $\cA^\infty$-valued oscillatory integral. We refer to~\cite[\S{4}]{HLP:Part1} for more detailed accounts on $\cA^\infty$-valued oscillatory integrals. Although the oscillatory integral is constructed only for maps with values in the \emph{smooth noncommutative torus} in~\cite{HLP:Part1} instead of $\cA^\infty$-valued maps, the construction in~\cite{HLP:Part1} holds \emph{verbatim} in our setting.
\end{remark}

\begin{remark}
The asymptotic expansions~(\ref{eq:PsiDOs.composition-asymptotic-expansion}) and~(\ref{eq:PsiDOs.adjoint-asymptotic-expansion}) are also mentioned in~\cite{LM:GAFA16}. However, the continuity assertions in Theorem~\ref{thm:PsiDOs.composition-formula} and Theorem~\ref{thm:PsiDOs.adjoint-formula} are not explicitly stated in~\cite{LM:GAFA16}. These continuity results can be proved by using $\cA^\infty$-valued oscillatory integrals along similar lines as in~\cite[Proposition~7.5 and Proposition 8.2]{HLP:Part2}. Note that the symbol of the adjoint operator~(\ref{eq:PsiDOs.adjoint-formula}) and its asymptotic expansion~(\ref{eq:PsiDOs.adjoint-asymptotic-expansion}) agree with the corresponding ones in the untwisted case mentioned in~\cite{Ba:CRAS88}. However, the composition formula~(\ref{eq:PsiDOs.composition-formula}) and its asymptotic expansion~(\ref{eq:PsiDOs.composition-asymptotic-expansion}) are different from the corresponding ones in the untwisted case mentioned in~\cite{Ba:CRAS88}. We refer to~\cite[Rem.\ 3.3]{LM:GAFA16} for the comparison of the asymptotic expansions of the composition products in the twisted and untwisted cases.
\end{remark}

Now we look for the asymptotic expansion of the composition product~(\ref{eq:PsiDOs.composition-formula}) of two classical symbols. To this end it is convenient to compute the asymptotic expansion of the last factor $\partial_x^\alpha|_{x=0} \big( \alpha_{-x} ( g(\xi+Bx) ) \big)$ of each summand of the expansion in~(\ref{eq:PsiDOs.composition-asymptotic-expansion}). Let $f:\Rn\rightarrow\cA^\infty$ be a smooth map. Note that we have
\begin{equation} \label{eq:PsiDOs.twisted-partial-derivative-order-one}
\partial_{x_j}\big|_{x=0} f(\xi+Bx) = \sum_{k=1}^n B_{kj}\partial_{\xi_k}f(\xi) =: \partial_{B,\xi_j} f(\xi) , \qquad j=1,\ldots,n .
\end{equation}
Given a multi-index $\alpha = (\alpha_1,\ldots,\alpha_n)\in\N_0^n$, let us denote $\partial_{B,\xi_1}^{\alpha_1}\cdots\partial_{B,\xi_n}^{\alpha_n}$ by $\partial_{B,\xi}^\alpha$. It then follows from~(\ref{eq:PsiDOs.twisted-partial-derivative-order-one}) that
\begin{equation} \label{eq:PsiDOs.twisted-higher-partial-derivative}
\partial_x^\alpha\big|_{x=0} f(\xi+Bx) = \partial_{B,\xi}^\alpha f(\xi) .
\end{equation}

Now let $f(\xi)\in\tCS^m(\Rn,\cA^\infty)$, $f(\xi)\sim\sum_{j\geq 0}f_{m-j}(\xi)$, be a classical symbol and $\alpha\in\N_0^n$. By using~(\ref{eq:PsiDOs.twisted-higher-partial-derivative}) we see that
\begin{align}
\nonumber f^{B,\alpha}(\xi) &:= \partial_x^\alpha\big|_{x=0} \alpha_{-x} \big( f(\xi+Bx) \big) \\\nonumber
&= \sum_{\beta+\gamma = \alpha} \binom \alpha \beta i^{|\beta|}\delta^\beta \big (\partial_x^\gamma\big|_{x=0} f(\xi+Bx) \big) \\
&= \sum_{\beta+\gamma = \alpha} \binom \alpha \beta i^{|\beta|}\delta^\beta \partial_{B,\xi}^\gamma f(\xi) . \label{eq:PsiDOs.twisted-partial-derivative-computation}
\end{align}
We know by Remark~\ref{rem:PsiDOs.classical-symbol-partial-derivative} that each summand $i^{|\beta|}\delta^\beta \partial_{B,\xi}^\gamma f(\xi)$ belongs to $\tCS^{m-|\gamma|}(\Rn,\cA^\infty)$ and
\begin{equation} \label{eq:PsiDOs.twisted-partial-derivative-each-summand-classical-symbol}
i^{|\beta|}\delta^\beta \partial_{B,\xi}^\gamma f(\xi)\sim\sum_{j\geq 0}i^{|\beta|}\delta^\beta \partial_{B,\xi}^\gamma f_{m-j}(\xi) ,
\end{equation}
in the sense of~(\ref{eq:PsiDOs.asymptotic-expansion-definition}). Combining this with~(\ref{eq:PsiDOs.twisted-partial-derivative-computation}) shows that $f^{B,\alpha}(\xi)$ is a classical symbol in $\tCS^m(\Rn,\cA^\infty)$. Furthermore, by using~(\ref{eq:PsiDOs.twisted-partial-derivative-each-summand-classical-symbol}) we also see that $f^{B,\alpha}(\xi)\sim\sum_{j\geq 0}f_{m-j}^{B,\alpha}(\xi)$ in the sense of~(\ref{eq:PsiDOs.asymptotic-expansion-definition}), where $f_{m-j}^{B,\alpha}(\xi)$, $j\geq 0$, are homogeneous symbols of degree $m-j$ given by
\begin{equation} \label{eq:PsiDOs.twisted-partial-derivative-homogeneous-parts}
f_{m-j}^{B,\alpha}(\xi) = \begin{cases} \displaystyle \sum_{k=0}^j \sum_{\substack{\beta+\gamma = \alpha \\ |\gamma| = k}} \binom \alpha \beta i^{|\beta|}\delta^\beta \partial_{B,\xi}^\gamma f_{m-j+k}(\xi) &\mbox{if } 0\leq j<|\alpha| \\
\displaystyle \sum_{k=0}^{|\alpha|} \sum_{\substack{\beta+\gamma = \alpha \\ |\gamma| = k}} \binom \alpha \beta i^{|\beta|}\delta^\beta \partial_{B,\xi}^\gamma f_{m-j+k}(\xi) & \mbox{if } j\geq|\alpha| \end{cases} .
\end{equation}
In particular, we have
\begin{equation*}
f_m^{B,\alpha}(\xi) = i^{|\alpha|}\delta^\alpha f_m(\xi) .
\end{equation*}
Summarizing the above discussion, we obtain the following lemma.

\begin{lemma} \label{lem:PsiDOs.twisted-partial-derivative-of-classical-symbol}
Let $f(\xi)\in\tCS^m(\Rn,\cA^\infty)$, $f(\xi)\sim\sum_{j\geq 0}f_{m-j}(\xi)$ and $\alpha\in\N_0^n$. Then $f^{B,\alpha}(\xi)$ given as in~\textup{(\ref{eq:PsiDOs.twisted-partial-derivative-computation})} belongs to $\tCS^m(\Rn,\cA^\infty)$ and $f^{B,\alpha}(\xi)\sim\sum_{j\geq 0}f_{m-j}^{B,\alpha}(\xi)$, where $f_{m-j}^{B,\alpha}(\xi)$, $j\geq 0$, are homogeneous symbols of degree $m-j$ given by~\textup{(\ref{eq:PsiDOs.twisted-partial-derivative-homogeneous-parts})}.
\end{lemma}

The following two results can be found in~\cite{HLP:Part1}.

\begin{proposition}[{\cite[Rem.\ 3.23]{HLP:Part1}}] \label{prop:PsiDOs.expanding-symbol-by-classical-symbols}
Let $f(\xi)\in\tS^m(\Rn,\cA^\infty)$ be such that $f(\xi)\sim\sum_{\ell\geq 0}f^{(\ell)}(\xi)$, where $f^{(\ell)}(\xi)\in\tCS^{m-\ell}(\Rn,\cA^\infty)$, $f^{(\ell)}(\xi)\sim\sum_{j\geq 0}f_{m-\ell-j}^{(\ell)}(\xi)$ is a classical symbol for all $\ell\geq 0$. Then $f(\xi)$ is a classical symbol as well and we have $f(\xi)\sim\sum_{j\geq 0}f_{m-j}(\xi)$ in the sense of Definition~\ref{def.PsiDOs.classical-symbols}, where $f_{m-j}(\xi)$, $j\geq 0$, is the homogeneous symbol of degree $m-j$ given by 
\begin{equation*}
f_{m-j}(\xi) = \sum_{\ell\leq j}f_{m-j}^{(\ell)}(\xi) .
\end{equation*}
\end{proposition}

\begin{proposition}[{\cite[Prop.\ 3.24]{HLP:Part1}}] \label{prop:PsiDOs.classical-symbol-product}
Let $f(\xi)\in\tCS^m(\Rn,\cA^\infty)$ with $f(\xi)\sim\sum_{p\geq 0}f_{m-p}(\xi)$ and $g(\xi)\in\tCS^{m'}(\Rn,\cA^\infty)$ with $g(\xi)\sim\sum_{r\geq 0}g_{m'-r}(\xi)$. Then $f(\xi)g(\xi)$ belongs to $\tCS^{m+m'}(\Rn,\cA^\infty)$, and we have $f(\xi)g(\xi)\sim\sum_{j\geq 0} (fg)_{m+m'-j}(\xi)$, where $(fg)_{m+m'-j}(\xi)$, $j\geq 0$, is the homogeneous symbol of degree $m+m'-j$ given by
\begin{equation*}
(fg)_{m+m'-j}(\xi) = \sum_{p+r=j}f_{m-p}(\xi)g_{m'-r}(\xi) .
\end{equation*}
\end{proposition}

Now we are in a position to compute the asymptotic expansion of the composition product~(\ref{eq:PsiDOs.composition-formula}) of two classical symbols.
\begin{theorem} \label{thm:PsiDOs.composition-formula-classical-symbols}
Let $f(\xi)\in\tCS^m(\Rn,\cA^\infty)$ with $f(\xi)\sim\sum_{j\geq 0}f_{m-j}(\xi)$ and $g(\xi)\in\tCS^{m'}(\Rn,\cA^\infty)$ with $g(\xi)\sim\sum_{j\geq 0}g_{m'-j}(\xi)$. Then the product $f\sharp g$ given by~\textup{(\ref{eq:PsiDOs.composition-formula})} is in $\tCS^{m+m'}(\Rn,\cA^\infty)$. Furthermore, $f\sharp g$ has the asymptotic expansion $f\sharp g(\xi)\sim\sum_{j\geq 0}(f\sharp g)_{m+m'-j}(\xi)$, where
\begin{equation} \label{eq:PsiDOs.classical-symbol-product-homogeneous-parts}
(f\sharp g)_{m+m'-j}(\xi) = \sum_{k+l+|\alpha| = j} \frac{(-i)^{|\alpha|}}{\alpha !} (\partial_\xi^\alpha f_{m-k})(\xi) g_{m'-l}^{B,\alpha}(\xi) , \qquad j\geq 0 .
\end{equation}
Here $g_{m'-l}^{B,\alpha}(\xi)$ is given as in~\textup{(\ref{eq:PsiDOs.twisted-partial-derivative-homogeneous-parts})}. In particular, we have
\begin{equation} \label{eq:PsiDOs.classical-symbol-product-principal-part}
(f\sharp g)_{m+m'}(\xi) = f_m(\xi) g_{m'}(\xi) .
\end{equation}
\end{theorem}

\begin{proof}
We know by Theorem~\ref{thm:PsiDOs.composition-formula} that the symbol $f\sharp g$ is in $\tS^{m+m'}(\Rn,\cA^\infty)$. Thus, it remains to prove that $f\sharp g$ is a classical symbol and its homogeneous parts are given by~(\ref{eq:PsiDOs.classical-symbol-product-homogeneous-parts}) and~(\ref{eq:PsiDOs.classical-symbol-product-principal-part}).

We know by Remark~\ref{rem:PsiDOs.classical-symbol-partial-derivative} and Lemma~\ref{lem:PsiDOs.twisted-partial-derivative-of-classical-symbol} that, for all $\alpha\in\N_0^n$, we have
\begin{gather*}
\partial_\xi^\alpha f(\xi)\sim\sum_{j\geq 0}\partial_\xi^\alpha f_{m-j}(\xi) \\
g^{B,\alpha}(\xi) := \partial_x^\alpha\big|_{x=0} \Big( \alpha_{-x} \big( g(\xi+Bx) \big) \Big) \sim\sum_{j\geq 0} g_{m'-j}^{B,\alpha}(\xi) ,
\end{gather*}
where $g_{m'-j}^{B,\alpha}(\xi)$ is given as in~(\ref{eq:PsiDOs.twisted-partial-derivative-homogeneous-parts}). It then follows from Proposition~\ref{prop:PsiDOs.classical-symbol-product} that
\begin{equation*}
\partial_\xi^\alpha f(\xi) \partial_x^\alpha\big|_{x=0} \Big( \alpha_{-x} \big( g(\xi+Bx) \big) \Big) \sim\sum_{j\geq 0}\sum_{k+l=j} \partial_\xi^\alpha f_{m-k}(\xi) g_{m'-l}^{B,\alpha}(\xi) .
\end{equation*}
Combining this with Proposition~\ref{prop:PsiDOs.expanding-symbol-by-classical-symbols} we obtain
\begin{equation*}
f\sharp g(\xi)\sim\sum_{j\geq 0}(f\sharp g)_{m+m'-j}(\xi) ,
\end{equation*}
where $(f\sharp g)_{m+m'-j}(\xi)$ is given by~(\ref{eq:PsiDOs.classical-symbol-product-homogeneous-parts}). In particular, we see that $(f\sharp g)_{m+m'}(\xi) = f_m(\xi) g_{m'}(\xi)$, and hence we get~(\ref{eq:PsiDOs.classical-symbol-product-principal-part}). The proof is complete.
\end{proof}

\begin{remark}
As we know by~(\ref{eq:PsiDOs.twisted-partial-derivative-homogeneous-parts}) that $g_{m'-j}^{B,0}(\xi) = g_{m'-j}(\xi)$ and $g_{m'-1}^{B,\alpha}(\xi) = i\delta^\alpha g_{m'-1}(\xi)$ for $|\alpha| = 1$, we see that, comparing with the corresponding term in the composition formula in the untwisted case, the term $(f\sharp g)_{m+m'-1}(\xi)$ remains unchanged.
\end{remark}

We close this section with the introduction of the trace on the algebra of pseudodifferential multipliers of order $<-n$. Let $\psi$ be an $\alpha$-invariant trace on $\cA$. Then, using the dual trace $\widehat{\psi}$ given in~(\ref{eq:PsiDOs.dual-trace-definition}), for $f\in\tS^m(\Rn,\cA^\infty)$, $m<-n$, we set
\begin{equation} \label{eq:PsiDOs.natural-trace-on-PsiDOs}
\Tr_\psi(P_f) := \widehat{\psi}\left( f^\vee \right) = \int_{\Rn} \psi\big( f(\xi) \big) \dbar\xi .
\end{equation}
This gives rise to a trace on $\bigcup_{m<-n}\tL_\sigma^m(\Rn,\cA^\infty)$ (\cf\ \cite[Rem.\ 3.4]{LM:GAFA16}).

\section{Boundedness} \label{sec:boundedness}
In this section, we study the boundedness of pseudodifferential multipliers with respect to the Hilbert $C^*$-norm. To this end, we first generalize the classical Schur's test on the boundedness of integral operators (see, e.g., \cite[Thm.\ 5.2]{HS:Springer78}).

Let us denote the space of continuous maps on $\Rn$ vanishing at infinity (resp., Schwartz class maps on $\Rn$) with values in $\cA$ by $C_0(\Rn,\cA)$ (resp., $\cS(\Rn,\cA)$). Let $\cH(\Rn,\cA)$ be the Hilbert $C^*$-module over $\cA$ formed by completing $\cS(\Rn,\cA^\infty)$ with respect to the norm $\|\cdot\| := \|\acou \cdot \cdot\|^\frac{1}{2}$, where $\acou \cdot \cdot$ is the inner product~(\ref{eq:PsiDOs.Cstar-module-inner-product}). Note that $\cH(\Rn,\cA)$ is nothing but the exterior tensor product $L^2(\Rn)\otimes\cA$ (\cf\ \cite[Chap.\ 4]{La:CUP95}), where $\cA$ is viewed as a Hilbert module over itself.

\begin{lemma} \label{lem:boundedness.Schur-test-positive-case}
Let $k:\Rn\times\Rn\rightarrow\cA$ be a continuous map with $k(x,y)\geq 0$ for all $x,y\in\Rn$. Furthermore, assume that there are positive measurable functions $p,q:\Rn\rightarrow\R_+$ and numbers $\alpha$ and $\beta$ such that
\begin{gather}
\int_{\Rn} k(x,y) q(y) dy \leq \alpha p(x) , \qquad \text{for $x\in\Rn$} , \label{eq:boundedness.positive-Schur-test-estimate1} \\
\int_{\Rn} k(x,y) p(x) dx \leq \beta q(y) , \qquad \text{for $y\in\Rn$} . \label{eq:boundedness.positive-Schur-test-estimate2}
\end{gather}
Then the \emph{integral operator} $:K:\cS(\Rn,\cA)\rightarrow C_0(\Rn,\cA)$ defined by
\begin{equation} \label{eq:boundedness.integral-operator-def-positive-case}
Ku(x) = \int_{\Rn} k(x,y)u(y) dy , \qquad u\in\cS(\Rn,\cA) ,
\end{equation}
extends by continuity to a bounded adjointable module endomorphism of the Hilbert $\cA$-module $\cH(\Rn,\cA)$ with $\|K\|\leq\sqrt{\alpha\beta}$.
\end{lemma}

\begin{proof}
The proof is essentially the same as in~\cite[Thm.\ 5.2]{HS:Springer78}. However, we recall and emphasize that for an element $u$ of $\cH(\Rn,\cA)$, although the integral $\int u(x)^* u(x) dx$ converges in $\cA$, in general $\int \|u(x)\|^2 dx$ is not necessarily finite. We thus prove the boundedness in the case $u\in\cS(\Rn,\cA)$ first and then extend the map to $\cH(\Rn,\cA)$ by continuity.

Given $u\in\cS(\Rn,\cA)$ we note that
\begin{equation} \label{eq:trick-square-root-of-the-kernel}
Ku(x) = \int_{\Rn} k(x,y)^{\frac{1}{2}}\sqrt{q(y)}k(x,y)^{\frac{1}{2}}\frac{u(y)}{\sqrt{q(y)}} dy = \acou {k(x,\cdot)^{\frac{1}{2}}\sqrt{q(\cdot)}} {k(x,\cdot)^{\frac{1}{2}}\frac{u(\cdot)}{\sqrt{q(\cdot)}}} . 
\end{equation}
Hence the Cauchy-Schwarz inequality for Hilbert $C^*$-modules (see, e.g., \cite[Prop.\ 1.1]{La:CUP95}) and~(\ref{eq:boundedness.positive-Schur-test-estimate1}) gives
\begin{multline*}
{\acou {k(x,\cdot)^{\frac{1}{2}}\sqrt{q(\cdot)}} {k(x,\cdot)^{\frac{1}{2}}\frac{u(\cdot)}{\sqrt{q(\cdot)}}}}^* \acou {k(x,\cdot)^{\frac{1}{2}}\sqrt{q(\cdot)}} {k(x,\cdot)^{\frac{1}{2}}\frac{u(\cdot)}{\sqrt{q(\cdot)}}} \\
\leq \left\| \int_{\Rn} k(x,y)q(y) dy \right\| \int_{\Rn} u(y)^*\frac{k(x,y)}{q(y)}u(y) dy \leq \alpha \int_{\Rn} u(y)^*\frac{p(x)k(x,y)}{q(y)}u(y) dy .
\end{multline*}
Combining this with~(\ref{eq:trick-square-root-of-the-kernel}) we obtain
\begin{align}
\nonumber \acou {Ku} {Ku} &= \int_{\Rn} {\acou {k(x,\cdot)^{\frac{1}{2}}\sqrt{q(\cdot)}} {k(x,\cdot)^{\frac{1}{2}}\frac{u(\cdot)}{\sqrt{q(\cdot)}}}}^* \acou {k(x,\cdot)^{\frac{1}{2}}\sqrt{q(\cdot)}} {k(x,\cdot)^{\frac{1}{2}}\frac{u(\cdot)}{\sqrt{q(\cdot)}}} dx \\
&\leq \alpha \int_{\Rn} \bigg( \int_{\Rn} u(y)^*\frac{p(x)k(x,y)}{q(y)}u(y) dy \bigg) dx . \label{eq:boundedness.Ku-estimate-before-Fubini}
\end{align}
Recall that the inequality $c^*ac \leq c^*bc$ holds for all $c\in\cA$ and self-adjoint elements $a$ and $b$ in $\cA$ such that $a\leq b$ (\cf\ \cite[Thm.\ 2.2.5]{Mu:AP90}). By combining this with the assumptions~(\ref{eq:boundedness.positive-Schur-test-estimate2}) and $u\in\cS(\Rn,\cA)$ we get
\begin{align*}
\int_{\Rn} \bigg( \int_{\Rn} u(y)^* \frac{p(x)k(x,y)}{q(y)} u(y) dx \bigg) dy &= \int_{\Rn} u(y)^* \int_{\Rn} \frac{p(x)k(x,y)}{q(y)} dx \, u(y)dy \\
&\leq \beta \int_{\Rn} u(y)^* u(y) dy \\
&\leq \beta \|u\|_{\cH(\Rn,\cA)}^2 < \infty .
\end{align*}
It follows from this that the integrand $u(y)^*p(x)k(x,y)q(y)^{-1}u(y)$ on the right-hand side of~(\ref{eq:boundedness.Ku-estimate-before-Fubini}) is integrable with respect to the product measure $dxdy$ and hence Fubini's theorem for the integration of vector-valued maps applies (\cf\ e.g., \cite[Prop.\ B.21]{HLP:Part1}). Thus, we get
\begin{align*}
\acou {Ku} {Ku} &\leq \alpha \int_{\Rn} \bigg( \int_{\Rn} u(y)^*\frac{p(x)k(x,y)}{q(y)}u(y) dx \bigg) dy \\
&\leq \alpha\beta \int_{\Rn} u(y)^* u(y) dy \\
&= \alpha\beta \acou u u .
\end{align*}
Furthermore, we have
\begin{equation*}
\|Ku\| = \| \acou {Ku} {Ku} \|^{\frac{1}{2}} \leq \sqrt{\alpha\beta} \| \acou u u \|^{\frac{1}{2}} = \sqrt{\alpha\beta} \|u\| .
\end{equation*}
This shows that the map $K:\cS(\Rn,\cA)\rightarrow C_0(\Rn,\cA)$ extends to a bounded linear map from $\cH(\Rn,\cA)$ to itself with $\|K\|\leq\sqrt{\alpha\beta}$. Furthermore, it is clear that $K$ is right $\cA$-linear and adjointable, whose adjoint is the continuous extension to $\cH(\Rn,\cA)$ of the integral operator associated with the kernel $\Rn\times\Rn\ni(x,y)\mapsto k(y,x)\in\cA$. This proves the lemma.
\end{proof}

\begin{remark}
We need an extension of Lemma~\ref{lem:boundedness.Schur-test-positive-case} to not necessarily non-negative continuous kernels $k:\Rn\times\Rn\rightarrow\cA$. One might be tempted to hope for a condition of the kind,
\begin{gather*}
\int_{\Rn} |k(x,y)|q(y) dy \leq \alpha p(x) , \qquad \text{for $x\in\Rn$} , \\
\int_{\Rn} |k(x,y)|p(x) dx \leq \beta q(y) , \qquad \text{for $y\in\Rn$} .
\end{gather*}
We leave it as an open problem whether such a version of Schur's test is valid for the Hilbert $C^*$-module $\cH(\Rn,\cA)$. The next weaker result, Lemma~\ref{lem:boundedness.Schur-test}, will suffice for our purposes.
\end{remark}

\begin{remark}
Needless to say in the previous and the next lemma we could replace the pair $(\Rn,dx)$ by any $\sigma$-finite measure space $(X,\mu)$ and the Hilbert module $\cH(\Rn,\cA)$ by $\cH(X,\cA,\mu) := L^2(X,\mu)\otimes\cA$.
\end{remark}

\begin{lemma} \label{lem:boundedness.Schur-test}
Let $k:\Rn\times\Rn\rightarrow\cA$ be continuous and suppose that there are positive measurable functions $p,q:\Rn\rightarrow\R_+$ and numbers $\alpha$ and $\beta$ such that
\begin{gather*}
\int_{\Rn} \|k(x,y)\|q(y) dy \leq \alpha p(x) , \qquad \text{for $x\in\Rn$} , \\
\int_{\Rn} \|k(x,y)\|p(x) dx \leq \beta q(y) , \qquad \text{for $y\in\Rn$} .
\end{gather*}
Then $Kf(x) = \int k(x,y)u(y) dy$ extends by continuity to a bounded adjointable $\cA$-module endomorphism of $\cH(\Rn,\cA)$ with $K^*f(x) = \int k(y,x)^* u(y) dy$ and $\|K\|\leq 4\sqrt{\alpha\beta}$.
\end{lemma}

\begin{proof}
Let $|k|,k^*:\Rn\times\Rn\rightarrow\cA$ denote the continuous maps defined by $|k|(x,y) = |k(x,y)|$ and $k^*(x,y) = k(x,y)^*$, $x,y\in\Rn$. We also set
\begin{equation*}
\Re{k} = \frac{1}{2}(k+k^*) \quad \text{and} \quad \Im{k} = \frac{1}{2i}(k-k^*) .
\end{equation*}
Then we may write $k = k_1-k_2+i(k_3-k_4)$, where $k_j:\Rn\times\Rn\rightarrow\cA$, $1\leq j\leq 4$, are the continuous maps defined by
\begin{equation*}
k_1 = \frac{1}{2}(|\Re{k}|+\Re{k}) , \quad k_2 = \frac{1}{2}(|\Re{k}|-\Re{k}) , \quad k_3 = \frac{1}{2}(|\Im{k}|+\Im{k}) , \quad k_4 = \frac{1}{2}(|\Im{k}|-\Im{k}) .
\end{equation*}
Note that $k_j(x,y)\geq 0$ for all $x,y\in\Rn$. Note also that any of the $k_j$ satisfies the assumptions~(\ref{eq:boundedness.positive-Schur-test-estimate1})--(\ref{eq:boundedness.positive-Schur-test-estimate2}) of Lemma~\ref{lem:boundedness.Schur-test-positive-case} as we have $\|k_j(x,y)\|\leq \|k(x,y)\|$ and hence
\begin{equation*}
\int_{\Rn} k_j(x,y)q(y) dy \leq \int_{\Rn} \|k_j(x,y)\|q(y) dy \leq \int_{\Rn} \|k(x,y)\| q(y) dy \leq \alpha p(x) \qquad \forall x\in\Rn ,
\end{equation*}
and similarly for $\int k_j(x,y) p(x) dx$. Thus, Lemma~\ref{lem:boundedness.Schur-test-positive-case} implies that
\begin{equation*}
\|K\| \leq \sum_{j=1}^4 \|K_j\| \leq 4\sqrt{\alpha\beta} ,
\end{equation*}
where $K_j$ denotes the respective integral operator associated with $k_j$ in the sense of~(\ref{eq:boundedness.integral-operator-def-positive-case}), $1\leq j\leq 4$. Therefore, $K$ extends to a bounded $\cA$-module endomorphism by continuity, and its adjoint $K^*$ is the continuous extension of the integral operator associated with the kernel $\Rn\times\Rn\ni(x,y)\mapsto k(y,x)^*$. The proof is complete.
\end{proof}

\begin{lemma} \label{lem:boundedness.hol-func-cal-of-symbols}
Let $f\in\tS^0(\Rn,\cA^\infty)$. Suppose that $f(\xi)$ is a normal element in $\cA$ for all $\xi\in\Rn$ and there are $a,b\in\R$ with $0<a<b$ such that $\Sp(f(\xi))\subset[a,b]$ for all $\xi\in\Rn$. Let $\Omega$ be an open subset of $\C$ containing $[a,b]$ and $\varphi$ be a holomorphic function on $\Omega$. Then the map $\xi\mapsto\varphi(f(\xi))$ belongs to $\tS^0(\Rn,\cA^\infty)$, where $\varphi(f(\xi))$, $\xi\in\Rn$, is defined by using holomorphic functional calculus.
\end{lemma}
\begin{proof}
Let $\Gamma$ be a rectifiable contour that winds once around $\bigcup_{\xi\in\Rn}\Sp(f(\xi))$. We define $\varphi(f(\xi))$, $\xi\in\Rn$, by using holomorphic functional calculus, i.e., we set
\begin{equation} \label{eq:boundedness.hol-func-cal-definition}
\varphi(f(\xi)) = \frac{1}{2\pi i}\int_\Gamma \varphi(z)(z-f(\xi))^{-1} dz .
\end{equation}
We know by~\cite{Co:Adv81} that $\cA^\infty$ is closed under holomorphic functional calculus. Therefore, we have $\varphi(f(\xi))\in\cA^\infty$ for all $\xi\in\Rn$.

By assumption $f(\xi)$ is a normal element in $\cA$ for all $\xi\in\Rn$, and so $(z-f(\xi))^{-1}$ is also normal for all $z\in\Gamma$ and $\xi\in\Rn$. As the spectral radius of a normal element $x\in\cA$ agrees with $\|x\|$ (see, e.g., \cite{Mu:AP90}), for each $z\in\Gamma$ and $\xi\in\Rn$, we have
\begin{align*}
\left\| (z-f(\xi))^{-1} \right\| &= \sup\left\{ |\lambda| ; \ \lambda\in\Sp\big( (z-f(\xi))^{-1} \big) \right\} \\
&= \sup\left\{ |z-\lambda|^{-1} ; \ \lambda\in\Sp\big( f(\xi) \big) \right\} .
\end{align*}
As $\Gamma$ is a compact subset of $\C$ and $\Sp(f(\xi))\subset[a,b]$ for all $\xi\in\Rn$, we see that there is $C>0$ such that $|z-\lambda|^{-1}\leq C$ for all $z\in\Gamma$ and $\lambda\in\bigcup_{\xi\in\Rn}\Sp(f(\xi))$. Thus, we get
\begin{equation} \label{eq:boundedness.resolvent-of-symbol-boundedness}
\left\| (z-f(\xi))^{-1} \right\| \leq C \qquad \text{for all $z\in\Gamma$ and $\xi\in\Rn$} .
\end{equation}
Let $\alpha,\beta\in\N_0^n$. Then the partial derivative $\delta^\alpha\partial_\xi^\beta[(z-f(\xi))^{-1}]$ can be written as a linear combination of terms of the form,
\begin{equation} \label{eq:boundedness.partial-derivative-of-resolvent-of-symbol}
(z-f(\xi))^{-1}\left( \delta^{\alpha^{(1)}}\partial_\xi^{\beta^{(1)}}f(\xi) \right)(z-f(\xi))^{-1}\cdots(z-f(\xi))^{-1}\left( \delta^{\alpha^{(l)}}\partial_\xi^{\beta^{(l)}}f(\xi) \right)(z-f(\xi))^{-1} ,
\end{equation}
where $\alpha^{(1)},\ldots,\alpha^{(l)}$ and $\beta^{(1)},\ldots,\beta^{(l)}$ are multi-orders such that $\alpha^{(1)}+\cdots+\alpha^{(l)} = \alpha$ and $\beta^{(1)}+\cdots+\beta^{(l)} = \beta$. As $f\in\tS^0(\Rn,\cA^\infty)$ we know that, for each $j=1,\ldots,l$, there is $C_j>0$ such that $\|\delta^{\alpha^{(j)}}\partial_\xi^{\beta^{(j)}} f(\xi)\|\leq C_j\brak{\xi}^{-|\beta^{(j)}|}$ for all $\xi\in\Rn$. Combining this with~(\ref{eq:boundedness.resolvent-of-symbol-boundedness}) and~(\ref{eq:boundedness.partial-derivative-of-resolvent-of-symbol}) shows that there is $C_{\alpha\beta}>0$ such that
\begin{equation} \label{eq:boundedness.partial-derivative-resolvent-estimate}
\left\| \delta^\alpha\partial_\xi^\beta \big[ (z-f(\xi))^{-1} \big] \right\| \leq C_{\alpha\beta}\brak{\xi}^{-|\beta|} \qquad \text{for all $z\in\Gamma$ and $\xi\in\Rn$} .
\end{equation}
It follows from this that the partial derivative $\partial_\xi^\beta$ and the contour integral on the right-hand side of~(\ref{eq:boundedness.hol-func-cal-definition}) can be swapped~\cite[Prop.\ C.28]{HLP:Part1}. Furthermore, as $\delta^\alpha:\cA^\infty\rightarrow\cA^\infty$ is a continuous linear map, $\delta^\alpha$ can also be interchanged with the same contour integral. Thus, by using~(\ref{eq:boundedness.partial-derivative-resolvent-estimate}) we obtain
\begin{align*}
\left\| \delta^\alpha\partial_\xi^\beta\big[ \varphi\big( f(\xi) \big) \big] \right\| &\leq \frac{1}{2\pi} \int_{r_0}^{r_1} \left| \varphi\big( \gamma(t) \big) \right| \, \left\| \delta^\alpha\partial_\xi^\beta\big[ (\gamma(t)-f(\xi))^{-1} \big] \right\| \, |\gamma'(t)| dt \\
&\leq \frac{C_{\alpha\beta}}{2\pi} \int_{r_0}^{r_1} \left| \varphi\big( \gamma(t) \big) \right| \, |\gamma'(t)| dt\cdot\brak{\xi}^{-|\beta|} ,
\end{align*}
where $\gamma:[r_0,r_1]\rightarrow\C$ is a parametrization of $\Gamma$. Thus, we see that $\varphi(f(\xi))\in\tS^0(\Rn,\cA^\infty)$. This proves the lemma.
\end{proof}
\begin{proposition} \label{prop:boundedness.PsiDOs-boundedness}
Let $f\in\tS^0(\Rn,\cA^\infty)$. Then the pseudodifferential multiplier $P_f$ gives rise to a continuous linear map from $\cH(\Rn,\cA)$ to itself. Furthermore, the map $f\mapsto P_f$ gives rise to a continuous linear map from $\tS^m(\Rn,\cA^\infty)$ to $\cL(\cH(\Rn,\cA))$ for every $m<0$.
\end{proposition}
\begin{proof}
Here we mimic the strategy of the proof of the boundedness of pseudodifferential operators on $\Rn$ using Schur's lemma in the literature (see, e.g., \cite{SR:CRC91}). However, as we are dealing with $\cA^\infty$-valued symbols, here we use Lemma~\ref{lem:boundedness.Schur-test} instead of the ordinary Schur's lemma (see, e.g., \cite[Lem.\ 3.7]{SR:CRC91}) and Lemma~\ref{lem:boundedness.hol-func-cal-of-symbols} instead of~\cite[Lem.\ 2.1]{SR:CRC91} for reducing the proof to the case of negative order symbols.
Let $u\in\cS(\Rn,\cA^\infty)$ and suppose that $f\in\tS^{-n-1}(\Rn,\cA^\infty)$. Using~(\ref{eq:PsiDOs:inverse-Fourier-transform-multiplier-formula}) and~(\ref{eq:PsiDOs.PsiDOs.definition}) we can write
\begin{align}
\nonumber P_fu(x) &= \iint e^{i\acou {x-y} \xi} \alpha_{-x}(f(\xi+Bx))u(y) dy\dbar\xi \\\nonumber
&= \int_{\Rn} \left( \int_{\Rn} e^{i \acou {x-y} \xi} \alpha_{-x}(f(\xi+Bx)) \dbar\xi \right) u(y) dy \\
&=: \int_{\Rn} K(x,y)u(y) dy . \label{eq:boundedness.PsiDO-integral-kernel-expression}
\end{align}
As $u\in\cS(\Rn,\cA^\infty)$ and $f\in\tS^{-n-1}(\Rn,\cA^\infty)$, all the integrands are absolutely integrable with respect to the norm $\|\cdot\|$ on $\cA$. Note also that, for every $\alpha\in\N_0^n$, we have
\begin{align*}
\left\| (x-y)^\alpha K(x,y) \right\| &= \left\| \int_{\Rn} (x-y)^\alpha e^{i\acou {x-y} \xi} \alpha_{-x}(f(\xi+Bx)) \dbar\xi \right \| \\
&= \left\| \int_{\Rn} e^{i\acou {x-y} \xi} \alpha_{-x}\big( \partial_\xi^\alpha f(\xi+Bx) \big) \dbar\xi \right\| \\
&\leq \int_{\Rn} \left\| \alpha_{-x} \big( \partial_\xi^\alpha f(\xi+Bx) \big) \right\| \dbar\xi \\
&= \int_{\Rn} \left\| \partial_\xi^\alpha f(\xi+Bx) \right\| \dbar\xi .
\end{align*}
We know that $\|\partial_\xi^\alpha f(\xi+Bx)\|\leq p_{|\alpha|}(f)\brak{\xi+Bx}^{-n-1-|\alpha|}$ for all $x,\xi\in\Rn$ since $f\in\tS^{-n-1}(\Rn,\cA^\infty)$. Here $p_{|\alpha|}$ is the semi-norm on the space of standard symbols given in~(\ref{eq:PsiDOs.standard-symbol-semi-norms}). It then follows that
\begin{equation*}
\left\| (x-y)^\alpha K(x,y) \right\| \leq p_{|\alpha|}(f) \int_{\Rn} \brak{\xi+Bx}^{-n-1-|\alpha|} \dbar\xi = p_{|\alpha|}(f) \int_{\Rn} \brak{\xi}^{-n-1-|\alpha|} \dbar\xi < \infty .
\end{equation*}
By using this we deduce that there is $C>0$ independent of $f$ such that
\begin{equation*}
(1+|x-y|^2)^n \|K(x,y)\|\leq C p_{2n}(f) \qquad \forall x,y\in\Rn .
\end{equation*}
It follows from this that there is a continuous semi-norm $p$ on $\tS^{-n-1}(\Rn,\cA^\infty)$ such that
\begin{equation*}
\sup_{x\in\Rn} \int_{\Rn} \left\| K(x,y) \right\| dy \leq p(f) , \qquad \sup_{y\in\Rn} \int_{\Rn} \left\|  K(x,y) \right\| dx \leq p(f) .
\end{equation*}
Combining this with Lemma~\ref{lem:boundedness.Schur-test} and~(\ref{eq:boundedness.PsiDO-integral-kernel-expression}) shows that
\begin{equation*}
\|P_f u\| \leq 4p(f)\|u\| \qquad \forall u\in\cS(\Rn,\cA^\infty) .
\end{equation*}
As $\cS(\Rn,\cA^\infty)$ is dense in $\cH(\Rn,\cA)$ this shows that $P_f$ uniquely extends to a bounded linear operator on $\cH(\Rn,\cA)$. Furthermore, the map $f\mapsto P_f$ gives rise to a continuous linear map from $\tS^{-n-1}(\Rn,\cA^\infty)$ to $\cL(\cH(\Rn,\cA))$.

We know by Theorem~\ref{thm:PsiDOs.composition-formula} and Theorem~\ref{thm:PsiDOs.adjoint-formula} that if $f\in\tS^{-(n+1)/2}(\Rn,\cA^\infty)$, then $f^\star\sharp f\in\tS^{-n-1}(\Rn,\cA^\infty)$. Using the Cauchy-Schwarz inequality for Hilbert $C^*$-modules  (see, e.g., \cite[Prop.\ 1.1]{La:CUP95}) and the proof in the case of order $-n-1$ symbols above we see that, for all $f\in\tS^{-(n+1)/2}(\Rn,\cA^\infty)$ and $u\in\cS(\Rn,\cA^\infty)$, we have
\begin{equation*}
\| P_f u \|^2 = \| \acou {P_f u} {P_f u} \| = \| \acou {P_{f^\star\sharp f}u} u \| \leq \| P_{f^\star\sharp f}u \| \|u\| \leq 4p(f^\star\sharp f)\|u\|^2  .
\end{equation*}
Combining this with the continuity assertions in Theorem~\ref{thm:PsiDOs.composition-formula} and Theorem~\ref{thm:PsiDOs.adjoint-formula} we see that there is a continuous semi-norm $q$ on $\tS^{-(n+1)/2}(\Rn,\cA^\infty)$ such that
\begin{equation*}
\| P_f u \| \leq q(f)\|u\| , \qquad \text{for $f\in\tS^{-(n+1)/2}(\Rn,\cA^\infty)$, $u\in\cS(\Rn,\cA^\infty)$} .
\end{equation*}
Since $\bigcup_{m<0}\tS^m(\Rn,\cA^\infty) = \bigcup_{k\in\N_0}\tS^{-(n+1)/{2^k}}(\Rn,\cA^\infty)$ we can proceed by induction to show that the map $f\mapsto P_f$ gives rise to a continuous linear map from $\tS^m(\Rn,\cA^\infty)$ to $\cL(\cH(\Rn,\cA))$ for every $m<0$.

Now suppose that $f\in\tS^0(\Rn,\cA^\infty)$. We set
\begin{equation*}
g(\xi) := \sup_{\xi\in\Rn}\|f(\xi)\|^2 - f(\xi)^*f(\xi) .
\end{equation*}
It follows from Lemma~\ref{lem:PsiDOs.pointwise-product-continuity} that $f(\xi)^*f(\xi)\in\tS^0(\Rn,\cA^\infty)$, and hence $g(\xi)$ belongs to $\tS^0(\Rn,\cA^\infty)$. Furthermore, we also see that, for all $\xi\in\Rn$, we have
\begin{equation*}
0\leq \sup_{\xi\in\Rn}\|f(\xi)\|^2 - \|f(\xi)^*f(\xi)\| \leq g(\xi) \leq \|g(\xi)\| \leq \sup_{\xi\in\Rn}\|f(\xi)\|^2 + \|f(\xi)^*f(\xi)\| \leq 2C < \infty ,
\end{equation*}
where we have set $C = \sup_{\xi\in\Rn}\|f(\xi)\|^2$ and $\leq$ is meant in the sense that, for $x,y\in\cA$, $x\geq y$ if $x-y\geq 0$ in $\cA$. Thus, we obtain
\begin{equation*}
1\leq 1+g(\xi)\leq \|1+g(\xi)\|\leq 1+2C \qquad \forall \xi\in\Rn .
\end{equation*}
This shows that $\Sp(1+g(\xi))\subset[1,1+2C]$. As $z\mapsto z^{\frac{1}{2}}$ is a holomorphic function on the domain containing $[1,1+2C]$ it follows from Lemma~\ref{lem:boundedness.hol-func-cal-of-symbols} that $h(\xi) := (1+g(\xi))^{\frac{1}{2}}\in\tS^0(\Rn,\cA^\infty)$.

We know by Theorem~\ref{thm:PsiDOs.composition-formula} and Theorem~\ref{thm:PsiDOs.adjoint-formula} that both $f^\star\sharp f(\xi)-f(\xi)^*f(\xi)$ and $h^\star\sharp h(\xi)-h(\xi)^*h(\xi)$ are in $\tS^{-1}(\Rn,\cA^\infty)$. Thus, there is $b\in\tS^{-1}(\Rn,\cA^\infty)$ such that
\begin{align*}
f^\star\sharp f(\xi) + h^\star\sharp h(\xi) &= f(\xi)^*f(\xi) + h(\xi)^*h(\xi) + b(\xi) \\
&= f(\xi)^*f(\xi) + 1 + g(\xi) + b(\xi) \\
&= f(\xi)^*f(\xi) + 1 + C - f(\xi)^*f(\xi) + b(\xi) \\
&= 1 + C + b(\xi) .
\end{align*}
Combining this with Theorem~\ref{thm:PsiDOs.composition-formula} and Theorem~\ref{thm:PsiDOs.adjoint-formula} once again we see that, for all $u\in\cS(\Rn,\cA^\infty)$, we have
\begin{equation} \label{eq:boundedness.Pfu-Cstar-algebra-ordering-estimate}
\acou {P_f u} {P_f u} \leq \acou {P_f u} {P_f u} + \acou {P_h u} {P_h u} = \acou {P_{f^\star\sharp f+h^\star\sharp h}u} u = (1+C)\acou u u + \acou {P_b u} u .
\end{equation}
Thanks to the above proof of the result in the case of negative order symbols and the Cauchy-Schwarz inequality for Hilbert $C^*$-modules we know that there is a continuous semi-norm $p$ on $\tS^{-1}(\Rn,\cA^\infty)$ such that
\begin{equation*}
\| \acou {P_b u} u \| \leq \| P_b u \| \|u\| \leq p(b) \|u\|^2 \qquad \forall u\in\cS(\Rn,\cA^\infty) .
\end{equation*}
Combining this with~(\ref{eq:boundedness.Pfu-Cstar-algebra-ordering-estimate}) we get
\begin{equation*}
\|P_f u\|^2 \leq (1+C) \|u\|^2 + \| \acou {P_b u} u \| \leq (1+C+p(b))\|u\|^2 \qquad \forall u\in\cS(\Rn,\cA^\infty) .
\end{equation*}
This shows that $P_f$ uniquely extends to a bounded linear operator from $\cH(\Rn,\cA)$ to itself. The proof is complete.
\end{proof}

\begin{remark}
The assertion about the continuity of the map $f\mapsto P_f$ in Proposition~\ref{prop:boundedness.PsiDOs-boundedness} will be utilized in the construction of the resolvent in \S\ref{sec:Resolvents}. For the sake of completeness, it is tempting to extend the continuity of the map $f\mapsto P_f$ to the space of symbols of degree $0$. However, due to the lack of continuity of the holomorphic functional calculus map $f(\xi)\mapsto\varphi(f(\xi))$ on $\tS^0(\Rn,\cA^\infty)$ used in the proof, the continuity of the map $\tS^0(\Rn,\cA^\infty)\ni f\mapsto P_f\in\cL(\cH(\Rn,\cA))$ cannot be obtained immediately from the proof of the boundedness of ordinary pseudodifferential operators in the literature. Thus, in this article, we only prove the continuity of the map $f\mapsto P_f$ on the space of symbols of negative orders, which suffices for our purpose. 
\end{remark}

\section{Weakly Parametric Pseudodifferential Calculus} \label{sec:weakly}
In this section, we introduce weakly parametric symbols and construct the weakly parametric pseudodifferential calculus in the setting of twisted $C^*$-dynamical systems.

In what follows we let $\Gamma$ be an open sector in $\C\setminus\{0\}$. Note that $\Gamma$ admits an exhaustion $\Gamma = \bigcup_{j\geq 0}\Gamma_j$, where the $\Gamma_j$ are closed subsectors of $\Gamma$ such that $\Gamma_j\subset\mathring{\Gamma}_{j+1}$.

By using Definition~\ref{def:PsiDOs:standard-symbols} of $\cA^\infty$-valued symbols, we can define weakly parametric $\cA^\infty$-valued symbols in the same way as in~\cite[Def.\ 1.1]{GS:IM95} as follows.

\begin{definition} \label{def:weakly.weakly-parametric-symbols}
Let $m,d\in\R$. The space $\tS^{m,0}(\Rn\times\Gamma,\cA^\infty)$ consists of maps $f(\xi,\mu)\in C^\infty(\Rn\times\Gamma,\cA^\infty)$ that are holomorphic with respect to $\mu\in\Gamma$ and satisfy, for all $j\geq 0$,
\begin{equation*}
\text{$\partial_z^j f(\cdot,\tfrac{1}{z})\in\tS^{m+j}(\Rn,\cA^\infty)$ for $\tfrac{1}{z}\in\Gamma$},
\end{equation*}
with uniform estimates in $\tS^{m+j}(\Rn,\cA^\infty)$ for $|z|\leq 1$ and $\tfrac{1}{z}$ in closed subsectors of $\Gamma$. Moreover, we set $\tS^{m,d}(\Rn\times\Gamma,\cA^\infty) = \mu^d\tS^{m,0}(\Rn\times\Gamma,\cA^\infty)$; that is, $\tS^{m,d}(\Rn\times\Gamma,\cA^\infty)$ consists of maps $f(\xi,\mu)\in C^\infty(\Rn\times\Gamma,\cA^\infty)$ that are holomorphic with respect to $\mu\in\Gamma$ such that, for all $j\geq 0$,
\begin{equation*}
\text{$\partial_z^j(z^d f(\cdot,\tfrac{1}{z}))\in\tS^{m+j}(\Rn,\cA^\infty)$ for $\tfrac{1}{z}\in\Gamma$},
\end{equation*}
with uniform estimates in $\tS^{m+j}(\Rn,\cA^\infty)$ for $|z|\leq 1$ and $\tfrac{1}{z}$ in closed subsectors of $\Gamma$. We call these symbols \emph{weakly parametric}.
\end{definition}


%
\begin{remark}
In Definition~\ref{def:weakly.weakly-parametric-symbols} we adapt the same notations and conventions given in~\cite[pp.\ 483--484]{GS:IM95}. For example, $\partial_z^j f(\xi,\frac{1}{z})$ means the $j$th $z$-derivative of the map $z\mapsto f(\xi,\frac{1}{z})$.
\end{remark}

We endow $\tS^{m,d}(\Rn\times\Gamma,\cA^\infty)$, $m,d\in\R$, with the locally convex topology generated by the semi-norms,
\begin{equation*}
p_{N\Gamma'}(f) := \sup_{|\alpha|+|\beta|+j\leq N} \sup_{\xi\in\Rn} \sup_{\substack{\frac{1}{z}\in\Gamma' \\ |z|\leq 1}} (1+|\xi|)^{-m-j+|\beta|} \left\| \delta^\alpha\partial_\xi^\beta\partial_z^j \big( z^d f \big( \xi,\tfrac{1}{z} \big) \big) \right\| ,
\end{equation*}
where $N$ ranges over all non-negative integers and $\Gamma'$ ranges over all closed subsectors of $\Gamma$. The space $\tS^{m,d}(\Rn\times\Gamma,\cA^\infty)$ is a Fr\'echet space with respect to these semi-norms. Note also that, if $m\leq m'$ and $d'-d\in\N_0$, then we have a continuous inclusion $\tS^{m,d}(\Rn\times\Gamma,\cA^\infty)\subset\tS^{m',d'}(\Rn\times\Gamma,\cA^\infty)$.

All the properties of weakly parametric pseudodifferential calculus and their proofs hold \emph{verbatim} in the setting of twisted $C^*$-dynamical systems, except for the composition formula (Theorem~\ref{thm:weakly.composition-formula} and Theorem~\ref{thm:weakly.composition-formula-weakly-polyhomogeneous}). Therefore, in the rest of this section, we only state the results and omit proofs if the same arguments work \emph{verbatim} in our setting.

\begin{lemma} \label{lem:weakly.pointwise-product-continuity}
Let $m_j,d_j\in\R$, $j=1,2$. Then the product of $\cA^\infty$ gives rise to a continuous bilinear map from $\tS^{m_1,d_1}(\Rn\times\Gamma,\cA^\infty)\times\tS^{m_2,d_2}(\Rn\times\Gamma,\cA^\infty)$ to $\tS^{m_1+m_2,d_1+d_2}(\Rn\times\Gamma,\cA^\infty)$.
\end{lemma}

Adopting the notation of~\cite{GS:IM95}, we shall denote
\begin{gather*}
\tS^{\infty,d}(\Rn\times\Gamma,\cA^\infty) := \bigcup_{m\in\R}\tS^{m,d}(\Rn\times\Gamma,\cA^\infty) , \\
\tS^{-\infty,d}(\Rn\times\Gamma,\cA^\infty) := \bigcap_{m\in\R}\tS^{m,d}(\Rn\times\Gamma,\cA^\infty) .
\end{gather*}

The following is the definition of asymptotic expansions, i.e., the analogue of~\cite[Def.\ 1.8]{GS:IM95}.

\begin{definition} \label{def:weakly.asymptotic-expansion-of-parametric-symbols}
Let $f(\xi,\mu)\in\tS^{m,d}(\Rn\times\Gamma,\cA^\infty)$, $m,d\in\R$ and $f_j(\xi,\mu)\in\tS^{m-j,d}(\Rn\times\Gamma,\cA^\infty)$, $j\geq 0$. We say that $f(\xi,\mu)\sim\sum_{j\geq 0}f_j(\xi,\mu)$ in $\tS^{\infty,d}(\Rn\times\Gamma,\cA^\infty)$ if
\begin{equation*}
f(\xi,\mu)-\sum_{j<N}f_j(\xi,\mu)\in\tS^{m-N,d}(\Rn\times\Gamma,\cA^\infty) \qquad \text{for all $N$} .
\end{equation*}
\end{definition}

The following is Borel's lemma for weakly parametric symbols, the counterpart of~\cite[Lem.\ 1.9]{GS:IM95} in the setting of twisted $C^*$-dynamical systems.

\begin{lemma} \label{lem:weakly.Borel-lemma}
Let $m,d\in\R$ and $f_j(\xi,\mu)\in\tS^{m-j,d}(\Rn\times\Gamma,\cA^\infty)$, $j\geq 0$. Then there exists $f(\xi,\mu)\in\tS^{m,d}(\Rn\times\Gamma,\cA^\infty)$ such that $f(\xi,\mu)\sim\sum_{j\geq 0}f_j(\xi,\mu)$ in $\tS^{\infty,d}(\Rn\times\Gamma,\cA^\infty)$.
\end{lemma}

The following definition of weakly polyhomogeneous symbols is the analogue of~\cite[Def.\ 1.10]{GS:IM95}. Here we only deal with $1$-step polyhomogeneous symbols, i.e., there is $m\in\R$ such that $m_j = m-j$ for all $j\geq 0$.

\begin{definition} \label{def:weakly.weakly-polyhomogeneous-symbols}
Let $d\in\R$. Then $f(\xi,\mu)\in\tS^{\infty,d}(\Rn\times\Gamma,\cA^\infty)$ is said to be \emph{weakly polyhomogeneous} if there exist symbols $f_{m-j}(\xi,\mu)\in\tS^{m-j-d,d}(\Rn\times\Gamma,\cA^\infty)$, $m\in\R$, $j=0,1,\ldots$, homogeneous in $(\xi,\mu)$ for $|\xi|\geq 1$ of degree $m-j$, such that $f(\xi,\mu)\sim\sum_{j\geq 0}f_{m-j}(\xi,\mu)$ in $\tS^{\infty,d}(\Rn\times\Gamma,\cA^\infty)$ in the sense of Definition~\ref{def:weakly.asymptotic-expansion-of-parametric-symbols}.
\end{definition}

\begin{remark} \label{rem:weakly.weakly-polyhomogeneous-symbol-partial-derivative}
Let $f(\xi,\mu)\in\tS^{\infty,d}(\Rn\times\Gamma,\cA^\infty)$, $f(\xi,\mu)\sim\sum_{j\geq 0}f_{m-j}(\xi,\mu)$ in $\tS^{\infty,d}(\Rn\times\Gamma,\cA^\infty)$, be a weakly polyhomogeneous symbol as in Definition~\ref{def:weakly.weakly-polyhomogeneous-symbols}. Then, for every $j\geq 0$ and $\alpha,\beta\in\N_0^n$, $\delta^\alpha\partial_\xi^\beta f_{m-j}(\xi,\mu)$ is in $\tS^{m-j-|\beta|-d,d}(\Rn\times\Gamma,\cA^\infty)$ and homogeneous in $(\xi,\mu)$ for $|\xi|\geq 1$ of degree $m-j-|\beta|$. Furthermore, it follows from the very definition of weakly polyhomogeneous symbols that, for all $\alpha,\beta\in\N_0^n$, $\delta^\alpha\partial_\xi^\beta f(\xi,\mu)$ is weakly polyhomogeneous and $\delta^\alpha\partial_\xi^\beta f(\xi,\mu)\sim\sum_{j\geq 0}\delta^\alpha\partial_\xi^\beta f_{m-j}(\xi,\mu)$ in $\tS^{\infty,d}(\Rn\times\Gamma,\cA^\infty)$ in the sense of Definition~\ref{def:weakly.asymptotic-expansion-of-parametric-symbols}.
\end{remark}

\begin{lemma} \label{lem:weakly.homogeneous-symbol-is-weakly-parametric}
Let $f:\Rn\times(\Gamma\cup\{0\})\rightarrow\cA^\infty$ be a smooth map such that $f$ is homogeneous in $(\xi,\mu)$ of degree $m\in\Z$ for $(|\xi|^2+|\mu|^2)^{\frac{1}{2}}\geq 1$ and holomorphic in $\mu\in\mathring{\Gamma}$. Then the following holds.
\begin{enumerate}
\item If $m\leq 0$, then $f\in\tS^{m,0}(\Rn\times\Gamma,\cA^\infty)\cap\tS^{0,m}(\Rn\times\Gamma,\cA^\infty)$.
\item For general $m$, $f\in\tS^{m,0}(\Rn\times\Gamma,\cA^\infty)+\tS^{0,m}(\Rn\times\Gamma,\cA^\infty)$ and $\delta^\alpha\partial_\xi^\beta f\in\tS^{m-|\beta|,0}(\Rn\times\Gamma,\cA^\infty)\cap\tS^{0,m-|\beta|}(\Rn\times\Gamma,\cA^\infty)$ when $|\beta|\geq m$.
\end{enumerate}
\end{lemma}

\begin{proof}
This is the analogue of~\cite[Lem.\ 1.14]{GS:IM95} for $\cA^\infty$-valued symbols, and this can be proved in the exactly same way as in the proof of~\cite[Lem.\ 1.14]{GS:IM95}.
\end{proof}

Given a multi-index $\alpha$, let us denote $\partial_{B,\xi_1}^{\alpha_1}\cdots\partial_{B,\xi_n}^{\alpha_n}$ by $\partial_{B,\xi}^\alpha$ as in~\S\S\ref{subsec:symbols-and-PsiDOs}. Let $f(\xi,\mu)\in\tS^{\infty,d}(\Rn\times\Gamma,\cA^\infty)$, $f(\xi,\mu)\sim\sum_{j\geq 0}f_{m-j}(\xi,\mu)$ in $\tS^{\infty,d}(\Rn\times\Gamma,\cA^\infty)$, be a weakly polyhomogeneous symbol as in Definition~\ref{def:weakly.weakly-polyhomogeneous-symbols} and $\alpha\in\N_0^n$. In the same way of deriving~(\ref{eq:PsiDOs.twisted-partial-derivative-computation}) we get
\begin{equation} \label{eq:weakly.twisted-partial-derivative-computation}
f^{B,\alpha}(\xi,\mu) := \partial_x^\alpha|_{x=0} \alpha_{-x} \big( f(\xi+Bx,\mu) \big) = \sum_{j=0}^{|\alpha|} \sum_{\substack{\beta+\gamma = \alpha \\ |\gamma| = j}} \binom \alpha \beta (i\delta)^\beta \partial_{B,\xi}^\gamma f(\xi,\mu) . 
\end{equation}
We know by Remark~\ref{rem:weakly.weakly-polyhomogeneous-symbol-partial-derivative} that each summand $(i\delta)^\beta \partial_{B,\xi}^\gamma f(\xi,\mu)$ is weakly polyhomogeneous and
\begin{equation*} 
(i\delta)^\beta \partial_{B,\xi}^\gamma f(\xi,\mu)\sim\sum_{j\geq 0}(i\delta)^\beta \partial_{B,\xi}^\gamma f_{m-j}(\xi,\mu) \qquad \text{in $\tS^{\infty,d}(\Rn\times\Gamma,\cA^\infty)$} ,
\end{equation*}
in the sense of Definition~\ref{def:weakly.asymptotic-expansion-of-parametric-symbols}. Combining this with~(\ref{eq:weakly.twisted-partial-derivative-computation}) shows that $f^{B,\alpha}(\xi,\mu)\in\tS^{\infty,d}(\Rn\times\Gamma,\cA^\infty)$ is a weakly polyhomogeneous symbol such that $f^{B,\alpha}(\xi,\mu)\sim\sum_{j\geq 0}f_{m-j}^{B,\alpha}(\xi,\mu)$ in $\tS^{\infty,d}(\Rn\times\Gamma,\cA^\infty)$ in the sense of Definition~\ref{def:weakly.asymptotic-expansion-of-parametric-symbols}, where for each $j\geq 0$, $f_{m-j}^{B,\alpha}(\xi,\mu)$ is in $\tS^{m-j-d,d}(\Rn\times\Gamma,\cA^\infty)$ and homogeneous in $(\xi,\mu)$ for $|\xi|\geq 1$ of degree $m-j$ given by
\begin{equation} \label{eq:twisted-partial-derivative-homogeneous-parts}
f_{m-j}^{B,\alpha}(\xi,\mu) = \begin{cases} \displaystyle \sum_{k=0}^j \sum_{\substack{\beta+\gamma = \alpha \\ |\gamma| = k}} \binom \alpha \beta (i\delta)^\beta \partial_{B,\xi}^\gamma f_{m-j+k}(\xi,\mu) &\mbox{if } 0\leq j<|\alpha| \\
\displaystyle \sum_{k=0}^{|\alpha|} \sum_{\substack{\beta+\gamma = \alpha \\ |\gamma| = k}} \binom \alpha \beta (i\delta)^\beta \partial_{B,\xi}^\gamma f_{m-j+k}(\xi,\mu) & \mbox{if } j\geq|\alpha| \end{cases} .
\end{equation}

Summarizing the above discussion, we obtain the following lemma.


%
\begin{lemma} \label{lem:twisted-partial-derivative-of-weakly-parametric-symbol}
Let $f(\xi,\mu)\in\tS^{\infty,d}(\Rn\times\Gamma,\cA^\infty)$, $f(\xi,\mu)\sim\sum_{j\geq 0}f_{m-j}(\xi,\mu)$ in $\tS^{\infty,d}(\Rn\times\Gamma,\cA^\infty)$, be a weakly polyhomogeneous symbol as in Definition~\ref{def:weakly.weakly-polyhomogeneous-symbols} and $\alpha\in\N_0^n$. Then the symbol,
\begin{equation*}
f^{B,\alpha}(\xi,\mu) := \partial_x^\alpha|_{x=0}\alpha_{-x} \big( f(\xi+Bx,\mu) \big) ,
\end{equation*}
satisfies $f^{B,\alpha}(\xi,\mu)\sim\sum_{j\geq 0}f_{m-j}^{B,\alpha}(\xi,\mu)$ in $\tS^{\infty,d}(\Rn\times\Gamma,\cA^\infty)$, where, for each $j\geq 0$, $f_{m-j}^{B,\alpha}(\xi,\mu)$ is in $\tS^{m-j-d,d}(\Rn\times\Gamma,\cA^\infty)$ and homogeneous in $(\xi,\mu)$ for $|\xi|\geq 1$ of degree $m-j$ given by~(\ref{eq:twisted-partial-derivative-homogeneous-parts}). In particular, we have $f_m^{B,\alpha}(\xi,\mu) = (i\delta)^\alpha f_m(\xi,\mu)$.
\end{lemma}

The following result is immediate from Definition~\ref{def:weakly.weakly-polyhomogeneous-symbols}.

\begin{proposition} \label{prop:expanding-symbol-by-weakly-polyhomogeneous-symbols}
Let $f(\xi,\mu)\in\tS^{m,d}(\Rn\times\Gamma,\cA^\infty)$ be such that $f(\xi,\mu)\sim\sum_{\ell\geq 0}f^{(\ell)}(\xi,\mu)$ in $\tS^{\infty,d}(\Rn\times\Gamma,\cA^\infty)$, where, for each $\ell\geq 0$, $f^{(\ell)}(\xi,\mu)\in\tS^{m-\ell,d}(\Rn\times\Gamma,\cA^\infty)$, $f^{(\ell)}(\xi,\mu)\sim\sum_{j\geq 0}f_{m-\ell-j}^{(\ell)}(\xi,\mu)$ in $\tS^{\infty,d}(\Rn\times\Gamma,\cA^\infty)$, is a weakly polyhomogeneous symbol such that for each $j\geq 0$, $f_{m-\ell-j}^{(\ell)}(\xi,\mu)$ is in $\tS^{m-\ell-j-d,d}(\Rn\times\Gamma,\cA^\infty)$ and homogeneous in $(\xi,\mu)$ of degree $m-\ell-j$ for $|\xi|\geq 1$. Then $f(\xi,\mu)$ is weakly polyhomogeneous as well and we have $f(\xi,\mu)\sim\sum_{j\geq 0}f_{m-j}(\xi,\mu)$, where, for each $j\geq 0$, $f_{m-j}(\xi,\mu)$ is defined by 
\begin{equation*}
f_{m-j}(\xi,\mu) := \sum_{\ell\leq j}f_{m-j}^{(\ell)}(\xi,\mu) , \qquad j\geq 0 .
\end{equation*}
Furthermore, $f_{m-j}(\xi,\mu)$ belongs to $\tS^{m-j-d,d}(\Rn\times\Gamma,\cA^\infty)$ and is homogeneous in $(\xi,\mu)$ of degree $m-j$ for $|\xi|\geq 1$.
\end{proposition}

We also have the following result.

\begin{proposition} \label{prop:weakly-polyhomogeneous-symbol-product}
Let $f(\xi,\mu)\in\tS^{\infty,d}(\Rn\times\Gamma,\cA^\infty)$, $f(\xi,\mu)\sim\sum_{p\geq 0}f_{m-p}(\xi,\mu)$, and $g(\xi,\mu)\in\tS^{\infty,d'}(\Rn\times\Gamma,\cA^\infty)$, $g(\xi,\mu)\sim\sum_{r\geq 0}g_{m'-r}(\xi,\mu)$ be weakly polyhomogeneous. Then $f(\xi,\mu)g(\xi,\mu)$ is weakly polyhomogeneous and we have $f(\xi,\mu)g(\xi,\mu)\sim\sum_{j\geq 0} (fg)_{m+m'-j}(\xi,\mu)$ in $\tS^{\infty,d+d'}(\Rn\times\Gamma,\cA^\infty)$, where $(fg)_{m+m'-j}(\xi,\mu)\in\tS^{m+m'-j-d-d',d+d'}(\Rn\times\Gamma,\cA^\infty)$, $j\geq 0$, are homogeneous symbols in $(\xi,\mu)$ for $|\xi|\geq 1$ of degree $m+m'-j$ given by
\begin{equation*}
(fg)_{m+m'-j}(\xi,\mu) = \sum_{p+r=j}f_{m-p}(\xi,\mu)g_{m'-r}(\xi,\mu) .
\end{equation*}
\end{proposition}

\begin{proof}
By the very definition of weakly polyhomogeneous symbols, we have
\begin{gather*}
f(\xi,\mu) = \sum_{p<N}f_{m-p}(\xi,\mu) \quad \text{mod} \,\, \tS^{m-N-d,d}(\Rn\times\Gamma,\cA^\infty) \\
g(\xi,\mu) = \sum_{r<N}g_{m'-r}(\xi,\mu) \quad \text{mod} \,\, \tS^{m'-N-d',d'}(\Rn\times\Gamma,\cA^\infty) .
\end{gather*}
From this we get
\begin{equation} \label{eq:weakly-polyhomogeneous-symbol-product-computation}
f(\xi,\mu)g(\xi,\mu) = \sum_{p,r<N}f_{m-p}(\xi,\mu)g_{m'-r}(\xi,\mu) \quad \text{mod} \,\, \tS^{m+m'-N-d-d',d+d'}(\Rn\times\Gamma,\cA^\infty) .
\end{equation}
As $f_{m-p}(\xi,\mu)$ (resp., $g_{m'-r}(\xi,\mu)$) is homogeneous in $(\xi,\mu)$ for $|\xi|\geq 1$ of degree $m-p$ (resp., degree $m'-r$), the symbol,
\begin{equation*}
(fg)_{m+m'-j}(\xi,\mu) := \sum_{p+r=j}f_{m-p}(\xi,\mu)g_{m'-r}(\xi,\mu)\in\tS^{m+m'-j-d-d',d+d'}(\Rn\times\Gamma,\cA^\infty) ,
\end{equation*}
is homogeneous in $(\xi,\mu)$ for $|\xi|\geq 1$ of degree $m+m'-j$ for all $j\geq 0$. Furthermore, since $f_{m-p}(\xi,\mu)\in\tS^{m-p-d,d}(\Rn\times\Gamma,\cA^\infty)$ and $g_{m'-r}(\xi,\mu)\in\tS^{m'-r-d',d'}(\Rn\times\Gamma,\cA^\infty)$ we see that $f_{m-p}(\xi,\mu)g_{m'-r}(\xi,\mu)$ belongs to $\tS^{m+m'-N-d-d',d+d'}(\Rn\times\Gamma,\cA^\infty)$ when $p+r\geq N$. Combining this with~(\ref{eq:weakly-polyhomogeneous-symbol-product-computation}) we see that, for all $N\geq 1$, we have
\begin{align*}
f(\xi,\mu)g(\xi,\mu) &= \sum_{j<N}\sum_{p+r=j}f_{m-p}(\xi,\mu)g_{m'-r}(\xi,\mu) \quad \text{mod} \,\, \tS^{m+m'-N-d-d',d+d'}(\Rn\times\Gamma,\cA^\infty) \\
&= \sum_{j<N}(fg)_{m+m'-j}(\xi,\mu) \quad \text{mod} \,\, \tS^{m+m'-N-d-d',d+d'}(\Rn\times\Gamma,\cA^\infty) .
\end{align*}
This shows that $f(\xi,\mu)g(\xi,\mu)$ is weakly polyhomogeneous and $f(\xi,\mu)g(\xi,\mu)\sim\sum_{j\geq 0}(fg)_{m+m'-j}(\xi,\mu)$ in $\tS^{\infty,d+d'}(\Rn\times\Gamma,\cA^\infty)$ in the sense of Definition~\ref{def:weakly.asymptotic-expansion-of-parametric-symbols}. The proof is complete.
\end{proof}

The following theorem is the analogue of~\cite[Thm.\ 1.12]{GS:IM95}.

\begin{theorem} \label{thm:asymptotic-expansion-of-symbol-in-lambda}
Let $m,d\in\R$. For $f\in\tS^{m,d}(\Rn\times\Gamma,\cA^\infty)$, set
\begin{equation*}
f_{(d,k)}(\xi) = \frac{1}{k!}\partial_z^k\Big( z^d f\big( \xi,\frac{1}{z} \big) \Big)|_{z=0} .
\end{equation*}
Then $f_{(d,k)}\in\tS^{m+k}(\Rn,\cA^\infty)$, and for any $N$, we have
\begin{equation*}
f(\xi,\mu) - \sum_{0\leq k<N} \mu^{d-k} f_{(d,k)}(\xi)\in\tS^{m+N,d-N}(\Rn\times\Gamma,\cA^\infty) .
\end{equation*}
\end{theorem}

\begin{proof}
In the proof of~\cite[Thm.\ 1.12]{GS:IM95}, the authors drop the variable $x$ for notational simplicity. If we regard $p(\xi,\frac{1}{z})$ in~\cite[Thm.\ 1.12]{GS:IM95} as an $\cA^\infty$-valued symbol in $\tS^{m,d}(\Rn\times\Gamma,\cA^\infty)$, then the proof of~\cite[Thm.\ 1.12]{GS:IM95} holds \emph{verbatim} in the case of $\cA^\infty$-valued symbols.
\end{proof}
Given $f\in\tS^{m,d}(\Rn\times\Gamma,\cA^\infty)$, $m,d\in\R$, we define the \emph{parametric pseudodifferential multiplier} associated with $f$ by
\begin{equation*}
(P_f u)(x) = \left( \int_{\Rn} f^\vee(y,\mu) \bU_y u \, dy \right) (x) , \qquad u\in\cS(\Rn,\cA^\infty) .
\end{equation*}
Here $f^\vee(y,\mu)$ is the inverse Fourier transform of $f(\xi,\mu)$ in the variable $\xi$.

Let $f\in\tS^{m,d}(\Rn\times\Gamma,\cA^\infty)$ and $g\in\tS^{m',d'}(\Rn\times\Gamma,\cA^\infty)$. For each $\mu\in\Gamma$, let $f\sharp g(\xi,\mu)$ be the composition product of $f(\xi,\mu)$ and $g(\xi,\mu)$ given by~(\ref{eq:PsiDOs.composition-formula}). As mentioned in~\cite{GS:IM95} the composition rule for a pseudodifferential calculus extends to the weakly parametric calculus in a straightforward way.  
The following theorems are the weakly parametric versions of Theorem~\ref{thm:PsiDOs.composition-formula} and Theorem~\ref{thm:PsiDOs.composition-formula-classical-symbols}.

\begin{theorem} \label{thm:weakly.composition-formula}
Let $f\in\tS^{m,d}(\Rn\times\Gamma,\cA^\infty)$ and $g\in\tS^{m',d'}(\Rn\times\Gamma,\cA^\infty)$. Then $f\sharp g(\xi,\mu)$ belongs to $\tS^{m+m',d+d'}(\Rn\times\Gamma,\cA^\infty)$ and we have $P_fP_g = P_{f\sharp g}$. Furthermore, $f\sharp g$ has the asymptotic expansion,
\begin{equation*}
f\sharp g(\xi,\mu)\sim\sum_\alpha \frac{(-i)^{|\alpha|}}{\alpha !} \partial_\xi^\alpha f(\xi,\mu) \partial_x^\alpha|_{x=0} \Big( \alpha_{-x} \big( g(\xi+Bx,\mu) \big) \Big) \quad \text{in $\tS^{\infty,d+d'}(\Rn\times\Gamma,\cA^\infty)$},
\end{equation*}
in the sense of Definition~\ref{def:weakly.asymptotic-expansion-of-parametric-symbols}.
\end{theorem}

\begin{theorem} \label{thm:weakly.composition-formula-weakly-polyhomogeneous}
Let $f(\xi,\mu)\in\tS^{\infty,d}(\Rn\times\Gamma,\cA^\infty)$, $f(\xi,\mu)\sim\sum_{j\geq 0}f_{m-j}(\xi,\mu)$, and $g(\xi,\mu)\in\tS^{\infty,d'}(\Rn\times\Gamma,\cA^\infty)$, $g(\xi,\mu)\sim\sum_{j\geq 0}g_{m'-j}(\xi,\mu)$ be weakly polyhomogeneous. Then $f\sharp g(\xi,\mu)$ is weakly polyhomogeneous and we have $P_fP_g = P_{f\sharp g}$. Furthermore, it admits the asymptotic expansion $f\sharp g(\xi,\mu)\sim\sum_{j\geq 0}(f\sharp g)_{m+m'-j}(\xi,\mu)$, where
\begin{equation} \label{eq:weakly-parametric-calculus-weakly-polyhomogeneous-symbol-product-homogeneous-parts}
(f\sharp g)_{m+m'-j}(\xi,\mu) = \sum_{k+l+|\alpha| = j} \frac{(-i)^{|\alpha|}}{\alpha !} \partial_\xi^\alpha f_{m-k}(\xi,\mu) g_{m'-l}^{B,\alpha}(\xi,\mu) , \qquad j\geq 0 .
\end{equation}
Here $g_{m'-l}^{B,\alpha}(\xi,\mu)$ is given as in~(\ref{eq:twisted-partial-derivative-homogeneous-parts}). In particular, we have
\begin{equation} \label{eq:weakly-parametric-calculus-weakly-polyhomogeneous-symbol-product-principal-part}
(f\sharp g)_{m+m'}(\xi,\mu) = f_m(\xi,\mu) g_{m'}(\xi,\mu) .
\end{equation}
\end{theorem}

\section{Asymptotic Expansions of Resolvents} \label{sec:Resolvents}
Let $\psi$ be an $\alpha$-invariant continuous trace on $\cA$. Then there is a natural trace $\Tr_\psi$ on the algebra of pseudodifferential multipliers of order $<-n$, which is given by~(\ref{eq:PsiDOs.natural-trace-on-PsiDOs}). In this section, we derive the asymptotic expansion of the trace $\Tr_\psi$ of weakly parametric pseudodifferential multipliers in the given parameter and apply it to derive the asymptotic expansion of the resolvent of a pseudodifferential multiplier which is elliptic with parameter.

\subsection{Asymptotic expansions of weakly parametric pseudodifferential multipliers}
In order to derive the asymptotic expansion of the trace of a weakly parametric pseudodifferential multiplier we need the following lemma.

\begin{lemma}[{\cite[Lem.\ 2.3]{GS:IM95}}] \label{lem:not-depending-on-the-argument}
Let $f(\mu)$ be a holomorphic function on a sector and suppose that
\begin{equation*}
f(\mu) = c(\theta) (re^{i\theta})^j \log^k(re^{i\theta}) + o \Big( r^j \log^k \Big( \frac{1}{r} \Big) \Big ) \qquad \text{as $r\rightarrow 0$} ,
\end{equation*}
where $r$ and $\theta$ are the modulus and argument of $\mu$, respectively. Then $c(\theta)$ is independent of $\theta$.
\end{lemma}

We are now in a position to get the asymptotic expansion of the trace of a weakly parametric pseudodifferential multiplier.

\begin{theorem} \label{thm:asymptotic-expansion-of-the-trace-of-weakly-parametric-psido}
Let $f(\xi,\mu)\in\tS^{\infty,d}(\Rn\times\Gamma,\cA^\infty)$, $f(\xi,\mu)\sim\sum_{j\geq 0}f_{m-j}(\xi,\mu)$ be weakly polyhomogeneous. Furthermore, assume that $f(\xi,\mu)$ and $f_{m-j}(\xi,\mu)$ with $m-j-d\geq -n$ are in $\tS^{m',d'}(\Rn\times\Gamma,\cA^\infty)$ with $m'<-n$. Then we have an asymptotic expansion,
\begin{equation} \label{eq:asymptotic-expansion-of-the-trace-of-weakly-parametric-psido}
\Tr_\psi(P_{f(\cdot,\mu)}) \sim \sum_{j=0}^\infty c_j \mu^{m-j+n} + \sum_{k=0}^\infty (c_k' \log\mu + c_k'')\mu^{-k+d} .
\end{equation}
\end{theorem}

\begin{proof} 
First, suppose that $d=0$. Given $J\in\N_0$, set $r_J(\xi,\mu) := f(\xi,\mu) - \sum_{0\leq j<J}f_{m-j}(\xi,\mu)$. Here we adapt the convention $r_0(\xi,\mu) = f(\xi,\mu)$. As $\psi$ is a continuous trace on the $C^*$-algebra $\cA$, there is $C>0$ such that
\begin{equation*}
|\psi(a)| \leq C \|a\| \qquad \forall a\in\cA ,
\end{equation*}
where $\|\cdot\|$ is the norm on the $C^*$-algebra $\cA$. Thus, we have
\begin{gather}
\label{eq:remainder-term-trace-estimate} |\psi(r_J(\xi,\mu))| \leq C \|r_J(\xi,\mu)\| \qquad \forall(\xi,\mu)\in\Rn\times\Gamma \\
\label{eq:homogeneous-terms-trace-estimate} |\psi(f_{m-j}(\xi,\mu))| \leq C \|f_{m-j}(\xi,\mu)\| \qquad \forall(\xi,\mu)\in\Rn\times\Gamma \,\, \forall j\geq 0 .
\end{gather}
By combining this with the assumption we see that, for each $\mu\in\Gamma$, $\psi(r_J(\xi,\mu))$ is integrable in $\xi$ and so are $\psi(f_{m-j}(\xi,\mu))$ for all $j\geq 0$.

Now consider the remainder $r_J(\xi,\mu)$. Using Theorem~\ref{thm:asymptotic-expansion-of-symbol-in-lambda} we get
\begin{equation*}
r_J(\xi,\mu) - \sum_{0\leq\nu<N}s_\nu(\xi)\mu^{-\nu} \in \tS^{m-J+N,-N}(\Rn\times\Gamma,\cA^\infty) ,
\end{equation*}
with $s_\nu(\xi)\in\tS^{m-J+\nu}(\Rn,\cA^\infty)$. Combining this with~(\ref{eq:remainder-term-trace-estimate}), for any $N$, we obtain
\begin{equation*}
\psi(r_J(\xi,\mu)) = \sum_{0\leq\nu<N}\psi(s_\nu(\xi))\mu^{-\nu} + O(\brak{\xi}^{m-J+N}\mu^{-N}) .
\end{equation*}
Given any $N$, choosing $J$ such that $m-J+N<-n$ ensures the integrability of each term $\psi(s_\nu(\xi))$. It then follows that
\begin{equation} \label{eq:trace-of-psido-associated-with-remainder-symbol}
\Tr_\psi(P_{r_J(\cdot,\mu)}) = \int_{\Rn} \psi(r_J(\xi,\mu)) \dbar\xi = \sum_{0\leq\nu<N}c_{NJ\nu}\mu^{-\nu} +O(\mu^{-N}) .
\end{equation}
Here the coefficients $c_{NJ\nu}$, which contribute to $c_k''$ in~(\ref{eq:asymptotic-expansion-of-the-trace-of-weakly-parametric-psido}), is defined by $\int_{\Rn}\psi(s_\nu(\xi))\dbar\xi$.

In order to compute the contribution of the homogeneous terms $f_{m-j}(\xi,\mu)$ as in the proof of~\cite[Thm.\ 2.1]{GS:IM95}, let us write $\Tr_\psi(P_{f_{m-j}(\cdot,\mu)})$ as the sum of three integrals as follows:
\begin{align}
\nonumber \Tr_\psi(P_{f_{m-j}(\cdot,\mu)}) &= \int_{\Rn}\psi(f_{m-j}(\xi,\mu)) \dbar\xi\\
&= \int_{|\xi|\geq|\mu|} \psi(f_{m-j}(\xi,\mu)) \dbar\xi + \int_{|\xi|\leq 1} \psi(f_{m-j}(\xi,\mu)) \dbar\xi + \int_{1\leq|\xi|\leq|\mu|} \psi(f_{m-j}(\xi,\mu)) \dbar\xi . \label{eq:splitting-the-integral-of-homogeneous-term}
\end{align}
For $|\mu|\geq 1$, by using the homogeneity of $f_{m-j}(\xi,\mu)$ for $|\xi|\geq 1$, we get
\begin{equation} \label{eq:the-first-integral-of-homogeneous-term}
\int_{|\xi|\geq|\mu|} \psi(f_{m-j}(\xi,\mu)) \dbar\xi = \mu^{m-j+n}\int_{|\xi|\geq 1} \Big( \frac{|\mu|}{\mu} \Big)^{m-j+n} \psi \Big( f_{m-j} \big( \xi,\frac{\mu}{|\mu|} \big) \Big) \dbar\xi .`
\end{equation}
This term will contribute to $c_j$ in~(\ref{eq:asymptotic-expansion-of-the-trace-of-weakly-parametric-psido}). Note that, at first glance, the right-hand side seems to depend on $\mu/|\mu|$, but this is not the case because $f_{m-j}(\xi,\mu)$ is holomorphic in $\mu$ and hence $\psi(f_{m-j}(\xi,\mu))$ is a holomorphic function in $\mu$ (\cf\ Lemma~\ref{lem:not-depending-on-the-argument}).

To compute the contribution of the second integral in~(\ref{eq:splitting-the-integral-of-homogeneous-term}), we apply Theorem~\ref{thm:asymptotic-expansion-of-symbol-in-lambda} to $f_{m-j}(\xi,\mu)$. Then we obtain
\begin{equation} \label{eq:homogeneous-symbol-expansion-in-parameter}
f_{m-j}(\xi,\mu) = \sum_{0\leq\nu<M}\mu^{-\nu}q_\nu(\xi) + R_M(\xi,\mu) .
\end{equation}
Here $q_\nu(\xi) := \frac{1}{\nu!}\partial_z^\nu f_{m-j}(\xi,\frac{1}{z})|_{z=0}$ is in $\tS^{m-j+\nu}(\Rn,\cA^\infty)$ and homogeneous of degree $m-j+\nu$ for $|\xi|\geq 1$. Furthermore, we also have
\begin{equation} \label{eq:remainder-degree-estimate}
R_M(\xi,\mu) = O(\brak{\xi}^{m-j+M}\mu^{-M}) , \qquad |\xi|\geq 1 .
\end{equation}
Thus, for the second term in~(\ref{eq:splitting-the-integral-of-homogeneous-term}), we get
\begin{equation} \label{eq:the-second-integral-of-homogeneous-term}
\int_{|\xi|\leq 1}\psi(f_{m-j}(\xi,\mu)) \dbar\xi = \sum_{0\leq\nu<M}\mu^{-\nu} \int_{|\xi|\leq 1} \psi(q_\nu(\xi)) \dbar\xi + O(\mu^{-M}) .
\end{equation}
Given any $N$ and its consequent choice of $J$ such that $m-J+N<-n$, we use the expansion~(\ref{eq:the-second-integral-of-homogeneous-term}) with $M\geq N$ to $f_{m-j}(\xi,\mu)$ for each $0\leq j<J$. This yields $J$ contributions to $c_k''$ for each $0\leq k<N$ in~(\ref{eq:asymptotic-expansion-of-the-trace-of-weakly-parametric-psido}).

We utilize the expansion~(\ref{eq:homogeneous-symbol-expansion-in-parameter}) again to compute the contribution of the third term in~(\ref{eq:splitting-the-integral-of-homogeneous-term}). Choose $M$ such that $M>-m+j-n$. As alluded to earlier $q_\nu(\xi)$ is homogeneous of degree $m-j+\nu$ for $|\xi|\geq 1$. Thus, by changing variables to polar coordinates we obtain
\begin{align}
\nonumber \mu^{-\nu}\int_{1\leq|\xi|\leq|\mu|} \psi(q_\nu(\xi))\dbar\xi &= \mu^{-\nu}c_\nu \int_1^{|\mu|} r^{m-j+\nu+n-1}dr \\
&= \begin{cases} \mu^{-\nu}c_\nu'(|\mu|^{m-j+\nu+n} - 1) &\mbox{if } m-j+\nu+n\neq 0 \\
\mu^{-\nu}c_\nu'\log{|\mu|} &\mbox{if } m-j+\nu+n=0 \end{cases} . \label{eq:third-integral-asymptotic-expansion-contribution-computation}
\end{align}
In~(\ref{eq:homogeneous-symbol-expansion-in-parameter}), note that, since $f_{m-j}(\xi,\mu)$ is homogeneous of degree $m-j$ in $(\xi,\mu)$ for $|\xi|\geq 1$ and each $q_\nu(\xi)$ is homogeneous of degree $m-j+\nu$ in $\xi$ for $|\xi|\geq 1$, it follows that $R_M(\xi,\mu)$ is homogeneous of degree $m-j$ in $(\xi,\mu)$ for $|\xi|\geq 1$. Let $R_M^h$ denote the extension of $R_M$ by homogeneity. Then, by using~(\ref{eq:remainder-degree-estimate}) we see that we have $R_M^h(\xi,\mu) = O(\brak{\xi}^{m-j+M}\mu^{-M})$ for all $\xi\neq 0$. This, together with the homogeneity of $R_M^h$ and the assumption $m-j+M>-n$, implies that
\begin{equation*}
\int_{|\xi|\leq|\mu|} \psi(R_M^h(\xi,\mu))\dbar\xi = c''\mu^{m-j+n} \quad \text{and} \quad \int_{|\xi|\leq 1} \psi(R_M^h(\xi,\mu))\dbar\xi = O(\mu^{-M}) ,
\end{equation*}
and hence
\begin{align*}
\int_{1\leq|\xi|\leq|\mu|} \psi(R_M(\xi,\mu))\dbar\xi &= \int_{|\xi|\leq|\mu|} \psi(R_M^h(\xi,\mu))\dbar\xi - \int_{|\xi|\leq 1} \psi(R_M^h(\xi,\mu))\dbar\xi \\
&= c''\mu^{m-j+n} - O(\mu^{-M}) .
\end{align*}
Combining this with~(\ref{eq:third-integral-asymptotic-expansion-contribution-computation}), the expansion~(\ref{eq:homogeneous-symbol-expansion-in-parameter}) and Lemma~\ref{lem:not-depending-on-the-argument} shows that the third integral in~(\ref{eq:splitting-the-integral-of-homogeneous-term}) can be written in the form,
\begin{equation} \label{eq:the-third-integral-of-homogeneous-term}
\int_{1\leq|\xi|\leq|\mu|} \psi(f_{m-j}(\xi,\mu)) \dbar\xi = (c+c'\log{\mu})\mu^{m-j+n} + \sum_{0\leq\nu<M}c_\nu\mu^{-\nu} + O(\mu^{-M}) .
\end{equation}
Here $c' = 0$ unless $m-j+n$ is an integer $\leq 0$.

Combining~(\ref{eq:the-first-integral-of-homogeneous-term}), (\ref{eq:the-second-integral-of-homogeneous-term}) and~(\ref{eq:the-third-integral-of-homogeneous-term}) we get
\begin{align}
\nonumber \Tr_\psi(P_{f_{m-j}(\cdot,\mu)}) &= \int_{\Rn}\psi(f_{m-j}(\xi,\mu)) \dbar\xi \\
&= c_j\mu^{m-j+n} + c_j'\mu^{m-j+n}\log{\mu} + \sum_{0\leq\nu<M}c_{j\nu}\mu^{-\nu} + O(\mu^{-M}) . \label{eq:asymptotic-expansion-of-the-trace-of-psido-associated-with-homogeneous-symbol}
\end{align}
By choosing $J$ such that $m-J+N<-n$ and $M$ satisfying $M\geq N$, the expansion~(\ref{eq:asymptotic-expansion-of-the-trace-of-weakly-parametric-psido}) for $d=0$ can be derived by using~(\ref{eq:trace-of-psido-associated-with-remainder-symbol}) and~(\ref{eq:asymptotic-expansion-of-the-trace-of-psido-associated-with-homogeneous-symbol}). This proves the theorem for $d=0$.

The result~(\ref{eq:asymptotic-expansion-of-the-trace-of-weakly-parametric-psido}) for general $d\in\R$ is immediate from the above proof since we can write $f(\xi,\mu) = \mu^d f'(\xi,\mu)$ and the asymptotic expansion of $f'(\xi,\mu)$ can be obtained from the above computation in the case $d=0$. This completes the proof.
\end{proof}

\subsection{Resolvents}
\begin{definition} \label{def:resolvents.elliptic-with-parameter}
Let $f\in\tS^m(\Rn,\cA^\infty)$, $f(\xi)\sim\sum_{j\geq 0}f_{m-j}(\xi)$ be a polyhomogeneous symbol. We say that $f$ is \emph{elliptic with parameter} $\mu\in\Gamma$ if it is elliptic of order $m$ and $f_m(\xi)-\mu^m$ is invertible in $\cA^\infty$ for all $\mu\in\Gamma$ and $|\xi|=1$. We also say that a pseudodifferential multiplier $P$ is \emph{elliptic with parameter} if it is associated with a polyhomogeneous symbol which is elliptic with parameter.
\end{definition}

\begin{theorem} \label{thm:parametrix-existence}
Let $m$ be a positive integer and a polyhomogeneous symbol $f\in\tS^m(\Rn,\cA^\infty)$, $f(\xi)\sim\sum_{j\geq 0}f_{m-j}(\xi)$ be elliptic with parameter. Then there is a weakly polyhomogeneous symbol $g(\xi,\mu)\in\tS^{-m,0}(\Rn\times\Gamma,\cA^\infty)\cap\tS^{0,-m}(\Rn\times\Gamma,\cA^\infty)$ such that:
\begin{enumerate}
\item $g(\xi,\mu)$ has the asymptotic expansions,
\begin{gather}
\label{eq:resolvent-symbol-expansion} g(\xi,\mu)\sim\sum_{j\geq 0}g_{-m-j}(\xi,\mu) \quad \text{in} \,\, \tS^{-m,0}(\Rn\times\Gamma,\cA^\infty) \cap \tS^{0,-m}(\Rn\times\Gamma,\cA^\infty) , \\
\label{eq:resolvent-symbol-without-principal-parti-expansion} g(\xi,\mu) - g_{-m}(\xi,\mu)\sim\sum_{j\geq 1}g_{-m-j}(\xi,\mu) \quad \text{in} \,\, \tS^{-m-1,0}(\Rn\times\Gamma,\cA^\infty) \cap \tS^{m-1,-2m}(\Rn\times\Gamma,\cA^\infty) .
\end{gather}
Here the homogeneous parts $g_{-m}(\xi,\mu)\in\tS^{-m,0}(\Rn\times\Gamma,\cA^\infty) \cap \tS^{0,-m}(\Rn\times\Gamma,\cA^\infty)$ and $g_{-m-j}(\xi,\mu)\in\tS^{-m-j,0}(\Rn\times\Gamma,\cA^\infty) \cap \tS^{m-j,-2m}(\Rn\times\Gamma,\cA^\infty)$, $j\geq 1$, are given by
\begin{gather}
\label{eq:resolvent-principal-part} g_{-m}(\xi,\mu) = (f_m(\xi)-\mu^m)^{-1} , \\
\label{eq:resolvent-jth-homogeneous-part} g_{-m-j}(\xi,\mu) = -\sum_{\substack{ k+l+|\alpha| = j \\ l<j }} \frac{(-i)^{|\alpha|}}{\alpha !} (f_m(\xi)-\mu^m)^{-1}\partial_\xi^\alpha f_{m-k}(\xi) g_{-m-l}^{B,\alpha}(\xi,\mu) , \qquad j\geq 1 .
\end{gather}
\item We have
\begin{equation*}
(f-\mu^m)\sharp g - 1 \in \tS^{-\infty,-m}(\Rn\times\Gamma,\cA^\infty) .
\end{equation*}
\end{enumerate}
\end{theorem}

\begin{proof}
First, we shall prove that $g_{-m-j}(\xi,\mu)$, $j\geq 0$, defined as in~(\ref{eq:resolvent-principal-part})--(\ref{eq:resolvent-jth-homogeneous-part}) belongs to $\tS^{-m,0}(\Rn\times\Gamma,\cA^\infty) \cap \tS^{0,-m}(\Rn\times\Gamma,\cA^\infty)$ for $j=0$ and belongs to $\tS^{-m-j,0}(\Rn\times\Gamma,\cA^\infty)\cap\tS^{m-j,-2m}(\Rn\times\Gamma,\cA^\infty)$ for $j\geq 1$. As $m$ is a positive integer we know by Lemma~\ref{lem:weakly.homogeneous-symbol-is-weakly-parametric} that $g_{-m}(\xi,\mu) = (f_m(\xi)-\mu^m)^{-1}$ belongs to $\tS^{-m,0}(\Rn\times\Gamma,\cA^\infty)\cap\tS^{0,-m}(\Rn\times\Gamma,\cA^\infty)$.

Now we proceed by induction to show that $g_{-m-j}(\xi,\mu)\in\tS^{-m-j,0}(\Rn\times\Gamma,\cA^\infty) \cap \tS^{m-j,-2m}(\Rn\times\Gamma,\cA^\infty)$ for all $j\geq 1$. Combining the result $g_{-m}(\xi,\mu) = (f_m(\xi)-\mu^m)^{-1}$ belongs to $\tS^{-m,0}(\Rn\times\Gamma,\cA^\infty)\cap\tS^{0,-m}(\Rn\times\Gamma,\cA^\infty)$ with the fact $\partial_\xi^\alpha f_{m-k}(\xi)\in\tS^{m-k-|\alpha|,0}(\Rn\times\Gamma,\cA^\infty)$ for all $\alpha\in\N_0^n$ and $g_{-m}^{B,\alpha}(\xi,\mu) = g_{-m}(\xi,\mu)$ for all $\alpha\in\N_0^n$ (\cf\ (\ref{eq:twisted-partial-derivative-homogeneous-parts})) shows that
\begin{equation*}
g_{-m-1}(\xi,\mu)\in\tS^{-m-1,0}(\Rn\times\Gamma,\cA^\infty) \cap \tS^{-1,-m}(\Rn\times\Gamma,\cA^\infty) \cap \tS^{m-1,-2m}(\Rn\times\Gamma,\cA^\infty) .
\end{equation*}
Suppose that, given an integer $j>1$, the symbols $g_{-m-l}(\xi,\mu)$ in~(\ref{eq:resolvent-jth-homogeneous-part}) belongs to $\tS^{-m-l,0}(\Rn\times\Gamma,\cA^\infty) \cap \tS^{-l,-m}(\Rn\times\Gamma,\cA^\infty) \cap \tS^{m-l,-2m}(\Rn\times\Gamma,\cA^\infty)$ for all $l<j$. Then~(\ref{eq:twisted-partial-derivative-homogeneous-parts}) implies that, for all $\alpha\in\N_0^n$, $g_{-m-l}^{B,\alpha}(\xi,\mu)$ belongs to the same symbol space. Furthermore, we also know that $(f_m(\xi)-\mu^m)^{-1}\in\tS^{-m,0}(\Rn\times\Gamma,\cA^\infty) \cap \tS^{0,-m}(\Rn\times\Gamma,\cA^\infty)$ and $\partial_\xi^\alpha f_{m-k}(\xi)\in\tS^{m-k-|\alpha|,0}(\Rn\times\Gamma,\cA^\infty)$ for all $\alpha\in\N_0^n$. Combining all this with~(\ref{eq:resolvent-jth-homogeneous-part}) shows that
\begin{equation*}
g_{-m-j}(\xi,\mu)\in\tS^{-m-j,0}(\Rn\times\Gamma,\cA^\infty) \cap \tS^{-j,-m}(\Rn\times\Gamma,\cA^\infty) \cap \tS^{m-j,-2m}(\Rn\times\Gamma,\cA^\infty) .
\end{equation*}
In particular, $g_{-m-j}(\xi,\mu)$ belongs to $\tS^{-m-j,0}(\Rn\times\Gamma,\cA^\infty) \cap \tS^{m-j,-2m}(\Rn\times\Gamma,\cA^\infty)$.

Observe that $(f_m(\xi)-\mu^m)^{-1}$ is homogeneous of degree $-m$ and $\partial_\xi^\alpha f_{m-k}(\xi)$ is homogeneous of degree $m-k-|\alpha|$ in $(\xi,\mu)$. Thereforem by using~(\ref{eq:twisted-partial-derivative-homogeneous-parts}), (\ref{eq:resolvent-principal-part})--(\ref{eq:resolvent-jth-homogeneous-part}) and an induction it follows that, for every $j\geq 0$, $g_{-m-j}(\xi,\mu)$ is homogeneous of degree $-m-j$ in $(\xi,\mu)$.

By Lemma~\ref{lem:weakly.Borel-lemma} there is a symbol $g(\xi,\mu)\in\tS^{-m,0}(\Rn\times\Gamma,\cA^\infty)\cap\tS^{0,-m}(\Rn\times\Gamma,\cA^\infty)$ satisfying~(\ref{eq:resolvent-symbol-expansion})--(\ref{eq:resolvent-symbol-without-principal-parti-expansion}). Recall that we have
\begin{equation*}
f(\xi)-\mu^m - \Big( f_m(\xi)-\mu^m + \sum_{1\leq j<N}f_{m-j}(\xi) \Big)\in\tS^{m-N,0}(\Rn\times\Gamma,\cA^\infty) \qquad \forall N\in\N ,
\end{equation*}
where we adopt the convention that $\sum_{1\leq j<N}f_{m-j}(\xi) = 0$ for $N=1$. Combining this with Theorem~\ref{thm:weakly.composition-formula-weakly-polyhomogeneous} and~(\ref{eq:resolvent-symbol-expansion})--(\ref{eq:resolvent-symbol-without-principal-parti-expansion}) we get
\begin{equation*}
(f-\mu^m)\sharp g-\sum_{j\geq 0}\big((f-\mu^m)\sharp g\big)_{-j}\in\tS^{-N,0}(\Rn\times\Gamma,\cA^\infty)\cap\tS^{m-N,-m}(\Rn\times\Gamma,\cA^\infty) \quad \forall N\in\N .
\end{equation*}
In particular, as we know by~(\ref{eq:weakly-parametric-calculus-weakly-polyhomogeneous-symbol-product-homogeneous-parts}) and~(\ref{eq:resolvent-principal-part})--(\ref{eq:resolvent-jth-homogeneous-part}) that $((f-\mu^m)\sharp g)_0(\xi,\mu) = 1$ and $((f-\mu^m)\sharp g)_{-j}(\xi,\mu) = 0$ for all $j\geq 1$, we have
\begin{equation*}
(f-\mu^m)\sharp g - 1\in\tS^{-\infty,-m}(\Rn\times\Gamma,\cA^\infty) .
\end{equation*}
This completes the proof.
\end{proof}

Let $f(\xi)$ and $g(\xi,\mu)$ be symbols as in Theorem~\ref{thm:parametrix-existence} and set $P = P_f$ and $Q(\mu) = P_{g(\cdot,\mu)}$. Then Theorem~\ref{thm:weakly.composition-formula-weakly-polyhomogeneous} and Theorem~\ref{thm:parametrix-existence} imply that $R(\mu): = (P-\mu^m)Q(\mu)-1 = P_{f\sharp g(\cdot,\mu)-1}$ is a pseudodifferential multiplier associated with a weakly polyhomogeneous symbol in $\tS^{-\infty,-m}(\Rn\times\Gamma,\cA^\infty)$. Along the same way as written in~\cite[p.\ 502]{GS:IM95} we can obtain the inverse of $P-\lambda := P-\mu^m$ by letting
\begin{equation} \label{eq:asymptotic.resolvent-construction}
(P-\lambda)^{-1} = Q(\lambda) + Q(\lambda)\sum_{j\geq 1}R(\lambda)^j .
\end{equation}
We know by Proposition~\ref{prop:boundedness.PsiDOs-boundedness} that the operator norm of $\cL(\cH(\Rn,\cA))$ of $R(\lambda)$ is $O(\lambda^{-1})$ for large $\mu$ in $\Gamma$. This ensures that the series in~(\ref{eq:asymptotic.resolvent-construction}) converges in the operator norm on $\cL(\cH(\Rn,\cA))$.

Once the resolvent $(P-\lambda)^{-1}$ is contructed, all the arguments for the derivation of the resolvent trace asymptotic~\cite[Thm.\ 2.7]{GS:IM95} holds \emph{verbatim} in our setting. Thus, we obtain the following result on the asymptotic expansion of the trace of the resolvent.

\begin{theorem} \label{eq:asymptotic.trace-expansion}
Let $P$ and $A$ be classical (i.e., $1$-step polyhomogeneous) pseudodifferential multipliers with respective orders $m\in\N$ and $\omega\in\R$. Suppose that $P$ is elliptic with parameter $\mu\in\Gamma$. Then, for $\lambda\in-\Gamma^m$ and $k$ with $-km+\omega<-n$, we have the asymptotic expansion,
\begin{equation*}
\Tr_\psi\left[ A(P-\lambda)^{-k} \right] \sim \sum_{j=0}^\infty c_j\lambda^{\frac{n+\omega-j}{m}-k} + \sum_{l=0}^\infty \big( c_l'\log{\lambda} + c_l'' \big) \lambda^{-k-l} .
\end{equation*}
Here the coefficients $c_j$, $c_l'$ and $c_l''$ are given by the the integral (over $\Rn$) of the trace $\psi$ of the respective symbols $f(\xi)\sim\sum_{j\geq 0}f_{m-j}(\xi)$ and $a(\xi)\sim\sum_{j\geq 0}a_{\omega-j}(\xi)$ of $P$ and $A$.
\end{theorem}

\begin{remark} \label{rem:asymptotic.two-traces-comparison}
As addressed in~\cite[\S{6}]{LM:GAFA16}, the trace $\Tr_\psi$ on $\bigcup_{m<-n}\tL_\sigma^m(\Rn,\cA^\infty)$ may not agree with the Hilbert space trace of a representation. Thus, in the case of parameter dependent symbols, the comparison of the two traces should be done as in~\cite[Thm.\ 6.2]{LM:GAFA16} in order to derive the asymptotic expansion of the Hilbert space trace of the operator $A(P-\lambda)^{-k}$. However, the argument in~\cite[\S{6}]{LM:GAFA16} for comparing the two traces is not applicable in the case of weakly parametric pseudodifferential calculus. The estimate in the proof of~\cite[Lem.\ 6.1]{LM:GAFA16}, which is a key ingredient in proving~\cite[Thm.\ 6.2]{LM:GAFA16}, is in complete analogy with the Shubin type parametric pseudodifferential calculus~\cite[\S{9}]{Sh:Springer01}. If $f(\xi,\lambda)$ is a Shubin type parametric symbol, the $\xi$-derivatives enhance the rate of decay in both $\xi$ and $\lambda$, which enables us to achieve the desired decay with respect to $\lambda$ in the proof of~\cite[Lem.\ 6.1]{LM:GAFA16}. However, in the case of weakly parametric symbols $f(\xi,\lambda)$, the $\xi$-derivatives do not change the rate of decay with respect to $\lambda$ (\cf\ \cite[Lem.\ 1.5]{GS:IM95}), and this is the reason why the method for comparing the two traces proposed in~\cite[\S{6}]{LM:GAFA16} cannot be applied to the case of weakly parametric calculus. We plan to address this problem in a future project where we want to conduct the comparison of the two traces and derive the asymptotic expansion of the Hilbert space trace of $A(P-\lambda)^{-k}$. The derivation of such an asymptotic expansion is important in view of the fact that the coefficient of the logarithmic term $\log{\lambda}$ in the expansion of the Hilbert space trace of $A(P-\lambda)^{-k}$ is essential in the study of the noncommutative residue trace (see, e.g., \cite{Le:AMS10} for a detailed account on this point).
\end{remark}

\end{document}